\documentclass[11pt,reqno,twoside]{article}

\usepackage{empheq}
\usepackage{graphicx}
\usepackage{float}
\usepackage{subfigure}
\usepackage{amssymb}
\usepackage{epstopdf}

\usepackage{hyperref}
\usepackage{bm}
\usepackage{amsmath,amsfonts,amsthm,mathrsfs,amssymb,cite}
\usepackage[usenames]{color}
\usepackage{array} 
\usepackage{paralist} 
\usepackage{verbatim} 
\usepackage{tabularx}
\usepackage{fancyhdr}
\usepackage{graphicx,subfigure,threeparttable}

\newtheorem{thm}{Theorem}[section]
\newtheorem{cor}[thm]{Corollary}
\newtheorem{lem}{Lemma}[section]
\newtheorem{prop}{Proposition}[section]

\theoremstyle{definition}
\newtheorem{defn}{Definition}[section]
\theoremstyle{remark}

\newtheorem{rem}{Remark}[section]
\numberwithin{equation}{section}

\newcommand{\norm}[1]{\left\Vert#1\right\Vert}

\newcommand{\bx}{\mathbf{x}}
\newcommand{\bnu}{\bm{\nu}}

\newcommand{\ri}{\mathrm{i}}

\newcommand{\abs}[1]{\left\vert#1\right\vert}

\newcommand{\by}{\mathbf{y}}
\newcommand{\bz}{\mathbf{z}}

\newcommand{\Acal}{\mathcal{A}}

\newcommand{\Lcal}{\mathcal{L}}

\newcommand{\Ocal}{\mathcal{O}}
\newcommand{\Tcal}{\mathcal{T}}

\title{The analysis of resonant frequencies and blow-up estimates of close-to-touching subwavelength resonators in the two-dimensional Helmholtz system}

\begin{document}
	
	\author{
		Hongjie Dong\footnote{Division of Applied Mathematics, Brown University, 182 George Street, Providence, RI 02912, United States of America.
			The work of this author was partially supported by the NSF under agreement DMS-2350129. (Hongjie\_Dong@brown.edu).}
		\and
		Hongjie Li\footnote{Yau Mathematical Sciences Center, Tsinghua University, Beijing, China. The work of this author was substantially supported by NSFC grant (12401561) and Research Start Fund (53331004324). (hongjieli@tsinghua.edu.cn; hongjie\_li@yeah.net).}
		\and
		Longjuan Xu\footnote{Academy for Multidisciplinary Studies, Capital Normal University, Beijing 100048, China.
			The work of this author was partially supported by NSF of China (12301141). (longjuanxu@cnu.edu.cn).}
	}
	\date{}
	\maketitle

	\begin{abstract}
		In this paper, we investigate wave scattering by a pair of closely spaced inclusions embedded in a homogeneous medium, characterized by a high contrast physical parameters. The system is modeled by the two-dimensional Helmholtz equation. We show that this configuration exhibits two sub-wavelength resonant modes, whose frequencies display distinct leading-order asymptotic behaviors. These findings differ significantly from those in the three-dimensional Helmholtz setting. Furthermore, we provide a quantitative analysis of the gradient blow-up rates for the wave field localized between the two resonators.
		
		\medskip 
		
		
	\end{abstract}
	
	\section{Introduction}
	
	When inclusions are embedded in a matrix material, where their physical parameters share a high contrast, they exhibit strong oscillations under a certain impinging wave. This phenomenon is known as the subwavelength resonance. The subwavelength denotes that the size of the embedded inclusions is smaller than the wavelength of the incident wave. Subwavelength resonances occur in various systems, including air bubbles in water\cite{AFGL007, DHB6012}, air bubbles in soft elastic materials\cite{CTL1646, LLZ0572}, hard inclusions in soft elastic materials \cite{LX9769, LZ2861, Liu00scie} and silicon particles in the air within electromagnetism \cite{ALL4825, KKL121}. Due to their resonant properties, even a small volume fraction of these particles can significantly alter wave propagation. Consequently, subwavelength resonators enable many important applications such as invisibility cloaking \cite{ACK6055, MN1715, QCJ2662, LLZ2555, L0069}, super-resolution \cite{BLLW9091, AZ0966, DLL1067, HLW1088}, and super-absorption \cite{LSM0301, LL9023}.
	
	When two resonators are brought close together, it is of significant interest to analyze their behavior, particularly the subwavelength resonant frequencies and the blow-up estimates of the wave field between them.
	For related physical studies, we refer the reader to \cite{HE5116, KBS0116, RAB9988} and the references therein. Mathematically, the work \cite{ADY2020, LZ2023, DHT2023} studied two adjacent bubbles in a liquid, modeled by the three-dimensional (3D) Helmholtz system.
	The authors in \cite{LX9769} investigated the case of two hard inclusions embedded in soft elastic materials governed by the 3D elastic system. However, to the best of our knowledge, there have been no studies on two-dimensional (2D) cases involving both the Helmholtz and elastic systems.
	In this paper, we focus on two adjacent resonators in the 2D Helmholtz system. The corresponding analysis in the 2D elasticity system is left for future work.
	Indeed, the 2D Helmholtz equation models acoustic wave propagation scattering from long cylinders or the transverse electromagnetic wave propagation corresponding to long cylindrical structures \cite{CK3518, LL0165}. 
	Related physical studies on pairs of cylinders in electrostatic and optical resonances can be found in \cite{MP7700}.
	Hence, it is both necessary and important to develop rigorous mathematical theories for two adjacent resonators governed by the 2D Helmholtz equation.

	In this paper, we consider the case that the two high-contrast particles are embedded in a matrix. The wave propagation is governed by the 2D Helmholtz equation. 
	As mentioned before, the model we considered here can apply to both acoustic and transverse electromagnetic wave propagation.
	Let $\delta$ denote the high-contrast physical parameter between the embedded inclusions and the surrounding matrix. Let $\varepsilon$ represent the distance between the two inclusions. In the 3D Helmholtz system, the corresponding configuration exhibits two subwavelength resonant frequencies \cite{ADY2020}.
	One resonant frequency is of the order $\Ocal(\sqrt{\delta\ln(1/\varepsilon)})$ and the other is of the order $\Ocal(\sqrt{\delta})$. In the 2D Helmholtz system, there are still two subwavelength resonant frequencies. 
	One resonant frequency is of the order $\Ocal(\sqrt{\delta/\sqrt{\epsilon}})$. 
	However, the other resonant frequency $\omega$ satisfies $\omega^2\ln \omega=\Ocal(\delta)$, which is a nonlinear equation and significantly different from the 3D case. 
	Moreover, the analysis is more intricate due to the greater complexity of the asymptotic expansion of the 2D Helmholtz fundamental solution compared to the 3D case. See Theorem \ref{thm:refre} for the details and Remark \ref{rem:diff3d} for further discussion.
	This analysis of resonant frequencies has important applications for designing metamaterials with multifrequency or broadband functionality in both acoustics and electromagnetics.
	
	Furthermore, we investigate the behavior of the eigenmodes as the resonators are brought closer. The work \cite{ADY2020, LZ2023} studied the the 3D Helmholtz system and showed that the maximum gradients for two wave fields are of the order $\Ocal(1/{(\varepsilon\ln\varepsilon)})$ and $\Ocal(1/\varepsilon)$, respectively. The authors in \cite{LX9769} proved the corresponding gradient estimates for the case of two hard inclusions embedded in soft elastic materials governed by the 3D elastic system. Therefore, it is natural to explore the gradient estimates for the wave fields localized between the two resonators in the 2D Helmholtz system. Prior studies in 2D electrostatics \cite{ABV2015, AKLLL2007, BLY2009, dfl2022, DL2019, G2015, KLY2015, KL2019, LY2009, MPM1988, P1989} and linear elasticity \cite{ACKLY, BLL2015, DLT3014, KY2019, L2021, LX2025,lz2020} have demonstrated that, when the material parameters of the inclusions tend to infinity, the electric field or stress field exhibits a characteristic blow-up rate of $1/\sqrt\varepsilon$. However, the blow-up mechanism in our present study differs fundamentally from these previous cases due to the resonant nature of the interaction.
	Unlike the static or non-resonant scenarios, where blow-up arises from extreme material contrasts, our analysis reveals a distinct singular behavior driven by the excitation of subwavelength resonances.  We derive rigorous gradient estimates in Theorems \ref{thm:blowup1} and \ref{thm:blowup}. It is noted that the gradient of one wave field depends on the values on the boundaries of the two resonators. More importantly, we establish both the upper and lower bounds of the gradient of the other wave field and show that the blow-up rate is of the order of $\Ocal(1/\varepsilon)$ at the narrow region between two resonators. This phenomenon highlights the unique interplay between geometric proximity and resonant coupling, leading to enhanced field concentrations. 
	
	The outline of the paper is as follows. The mathematical set-up of the considered problem is given in Section \ref{sec-setup}. In Section \ref{sec-pre}, we provide some preliminaries for our subsequent use. Section \ref{sec-resonance} is devoted to the resonant analysis of the structure. Finally, in Section \ref{sec:blowup}, we study the blow-up estimates of the gradient of the waves between two adjacent resonators. 
	
	\section{Mathematical set-up}\label{sec-setup}
	
	In this section, we present the mathematical set-up of the considered problem. Then, the definitions of the resonant phenomenon and eigenmodes are provided from a mathematical perspective. 
	
	We consider two $C^{2,\alpha}$ inclusions (resonators) $D_1$ and $D_2$ embedded in a matrix, where $\alpha\in(0,1)$. Denote $D:=D_1 \cup D_2$. Let $\rho_b$ and $\kappa_b$ signify the density and bulk modulus of the resonators in acoustics, respectively. The corresponding parameters in the matrix are $\rho$ and $\kappa$. We then introduce the following auxiliary parameters 
	\begin{equation}\label{eq:aux-param}
		v = \sqrt{\frac{\kappa}{\rho}}, \quad v_b = \sqrt{\frac{\kappa_b}{\rho_b}}, \quad k = \frac{\omega}{v}, \quad k_b = \frac{\omega}{v_b},
	\end{equation}
	which are wave speeds and wave numbers in $\mathbb{R}^2\setminus \overline{D}$ and in $D$, respectively. We also introduce the dimensionless contrast parameters 
	\begin{equation}\label{eq:dim-contrast}
		\delta = \frac{\rho_b}{\rho}, \quad \tau = \frac{k_b}{k} = \frac{v}{v_b} = \sqrt{\frac{\rho_b \kappa}{\rho \kappa_b}}.
	\end{equation}
	We assume that $v$, $v_b$, and $\tau$ are all $\Ocal(1)$, and there is a high contrast between the densities, namely 
	\[
	\delta\ll 1. 
	\]
	Indeed, when bubbles are immersed in liquids, the related parameters satisfy the assumptions here \cite{DHB6012}. For the corresponding physical meaning of the parameters in the transverse electromagnetic, we refer the reader to \cite{CK3518, LL0165}.

	Let $u^{in}(\bx)$ be an incident wave satisfying the equation
	\[
	(\Delta+k^2)u^{in}=0 \quad \text{in} \quad  \mathbb{R}^2.
	\]
	Then the total wave produced by the above system is controlled by the system 
	\begin{align}\label{eq:or1}
		\begin{cases}
			(\Delta+k^2)u=0& \quad \text{in} \quad  \mathbb{R}^2 \setminus\overline{D},\\
			(\Delta +k_b^2) u=0&  \quad \text{in} \quad  D,\\
			u |_+-u|_-=0& \quad \text{on}\quad \partial D, \\
			\delta \frac{\partial u}{\partial \bnu}|_+-\frac{\partial u}{\partial \bnu}|_- =0& \quad \text{on}\quad \partial D,
		\end{cases}
	\end{align}
	where $\bnu$ denotes the exterior unit normal vector to $\partial D$, and the subscripts $\pm$ indicate the limits from outside and inside of $D$, respectively. In the above system, the scattering wave $u^s:=u-u^{in}$ satisfies the Sommerfeld radiation condition, namely
	\[
	\lim_{\abs{\bx}\rightarrow\infty} \sqrt{\abs{\bx}}\left(\frac{\partial u^s}{\partial \abs{\bx}} - \ri k u^s \right) = 0,
	\]
	where $\ri$ denotes the imaginary unit.
	
	We will use the potential theory to analyze the resonant phenomenon of the system \eqref{eq:or1}. Thus, we first present the related results for the subsequent use. 
	The fundamental solution of the operator $\Delta +k^2$ is given by 
	\begin{equation}\label{eq:fuda1}
		G^k(\bx,\by)=-\frac{\ri}{4}H_0^{(1)}(k|\bx-\by|),
	\end{equation}
	where $H_0^{(1))}$ is the Hankel function of the first kind and zero order \cite{CK3518, WGB0731}. Then the single-layer potential operator associated with the domain $D$ is defined by 
	\[
	S^k_D[\varphi](\bx) := \int_{\partial D} G^k(\bx - \by) \varphi(\by)ds(\by), \quad \bx\in \partial D,~\varphi\in L^2(\partial D).
	\]
	On the boundary $\partial D$, the conormal derivative of the single-layer potential satisfies the following jump formula
	\begin{equation}\label{eq:jump}
		\partial_{\bnu}   S_{ D}^{k}[\varphi]|_{\pm}(\bx)=\left( \pm\frac{1}{2} I +  K_{ D}^{k, *} \right)[\varphi](\bx), \quad \bx\in\partial D,
	\end{equation}
	where
	\[
	K_{ D}^{k, *} [\varphi](\bx)=\mbox{p.v.} \int_{\partial D} \partial_{\bnu_{\bx}} G^{k}(\bx-\by)\varphi(\by)ds(\by)
	\]
	with $\mbox{p.v.}$ standing for the Cauchy principal value. It is noted that the operator $ K_{ D}^{k, *}$ below \eqref{eq:jump} is called the Neumann-Poincar\'e (N-P) operator.
	
	With the above preparations, the solution to the system \eqref{eq:or1} can be written as 
	\begin{equation*}
		u(\bx)=\begin{cases}
			u^{in}+S_D^{k}[\psi]  \quad  & \bx\in \mathbb{R}^2\setminus\overline{D},\\
			S_D^{k_b}[\varphi] \quad  & \bx\in D
		\end{cases}
	\end{equation*}
	for some surface potentials $(\varphi, \psi)\in L^2(\partial D) \times L^2(\partial D)$. 
	From the third and fourth conditions in \eqref{eq:or1}, the surface potentials $(\varphi, \psi)$ should satisfy 
	\begin{equation}\label{eq:or2}
		\Acal(\omega ,\delta) \begin{bmatrix}
			\varphi\\
			\psi 
		\end{bmatrix}=\begin{bmatrix}
			u^{in}\\
			\delta \frac{\partial u^{in}}{\partial \bnu}
		\end{bmatrix},
	\end{equation}
	where the operator $A(\omega ,\delta)$ is given by 
	$$
	\Acal(\omega,\delta )=\begin{pmatrix}
		S_D^{k_b} &  -S_D^k \\
		-\frac{1}{2}I+K_D^{k_b,*} & -\delta (\frac{1}{2}I+K_D^{k,*})
	\end{pmatrix}.
	$$
	In view of \eqref{eq:or2}, we can define the resonance of the system \eqref{eq:or1}.
	\begin{defn}
		We say that the resonance occurs for the system \eqref{eq:or1}, if there exists a resonant frequency $\omega\in\mathbb{C}$ with positive real part and negative imaginary part such that the operator $\Acal$ has a nontrivial kernel, where the operator $\Acal$ is given in \eqref{eq:or2}. In other words, there exists a nontrivial solution $(\varphi, \psi)\in L^2(\partial D) \times L^2(\partial D)$ such that 
		\[
		\Acal(\omega ,\delta) \begin{bmatrix}
			\varphi\\
			\psi 
		\end{bmatrix}=0.
		\]
		The functions $(\varphi, \psi)$ are called the resonant eigenfunctions. 
		For such a resonant frequency, we define the corresponding resonant mode (or eigenmode) as 
		\begin{equation}\label{eq:deeimo}
			u(\bx)=\begin{cases}
				S_D^{k_b}[\varphi] \quad  & \bx\in D,\\
				S_D^{k}[\psi]  \quad  & \bx\in \mathbb{R}^2\setminus\overline{D}.
			\end{cases}
		\end{equation}
	\end{defn}
	
	\begin{rem}
		The resonant mode defined in \eqref{eq:deeimo} is normalized in the sense that $$\norm{u}_{L^2(\partial D)} = \Ocal(1).$$ 
	\end{rem}

	\section{Preliminaries}\label{sec-pre}
	In this section, we provide some preliminaries and derive some auxiliary results for our subsequent use. 
	
	The fundamental solution $G^k(\bx,\by)$ defined in \eqref{eq:fuda1} has the asymptotic expansion for $k\ll 1$ \cite{AFGL007},
	\begin{equation}\label{eq:fuaym}
		G^k(\bx,\by)= \frac{1}{2\pi} \ln |\bx-\by| + \eta_k + \sum_{j=1}^{\infty} (b_j \ln k |\bx-\by| + c_j)(k|\bx-\by|)^{2j},
	\end{equation}
	where 
	\begin{equation}\label{def:eta}
		\eta_k = \frac{1}{2\pi} (\ln k + \gamma - \ln 2) - \frac{\ri}{4}, \quad b_j = \frac{(-1)^j}{2\pi} \frac{1}{2^{2j} (j!)^2}, \quad c_j = b_j \left( \gamma - \ln 2 - \frac{\ri\pi}{2} - \sum_{n=1}^{j} \frac{1}{n} \right),
	\end{equation}
	and \(\gamma\) is the Euler constant. In particular,
	\[
	b_1 = -\frac{1}{8\pi}, \quad c_1 = -\frac{1}{8\pi} (\gamma - \ln 2 - 1 - \frac{\ri\pi}{2}).
	\]
	Thus, from \eqref{eq:fuaym} the following asymptotic expansion holds for the single-layer potential
	\begin{equation}\label{eq:assin}
		S_D^k = \hat{S}_D^k + \sum_{j=1}^{\infty} \left(k^{2j} \ln k\right) S_{D,j}^{(1)} + \sum_{j=1}^{\infty} k^{2j} S_{D,j}^{(2)},
	\end{equation}
	where
	\begin{equation}\label{def:SD}
		S_D [\psi](\bx)= \frac{1}{2\pi}\int_{\partial D} \ln |\bx-\by| \psi(\by)ds(\by) , \quad \hat{S}_{D}^k [\psi](\bx) = S_D [\psi](\bx) + \eta_k \int_{\partial D} \psi(\by) \, ds(\by),
	\end{equation}
	and
	\[
	S_{D,j}^{(1)} [\psi](\bx) = \int_{\partial D} b_j |\bx-\by|^{2j} \psi(\by) ds(\by),
	\]
	\[
	S_{D,j}^{(2)}[\psi](\bx) = \int_{\partial D} |\bx - \by|^{2j} (b_j \ln |\bx - \by| + c_j) \psi(\by) \, ds(\by).
	\]
	For the N-P operator \(K_D^{k,*}\) defined below \eqref{eq:jump}, the following asymptotic expansion holds
	\begin{equation}\label{eq:asnp}
		K_D^{k,*} = K_D^* + \sum_{j=1}^\infty \left( k^{2j} \ln k \right) K_{D,j}^{(1)} + \sum_{j=1}^\infty k^{2j} K_{D,j}^{(2)},
	\end{equation}
	where
	\begin{equation}\label{def-KDj}
		\begin{split}
			K_{D,j}^{(1)} [\psi](\bx) & = \int_{\partial D} b_j \frac{\partial |\bx-\by|^{2j}}{\partial \bnu(\bx)} \psi(\by) \, ds(\by), \\
			K_{D,j}^{(2)} [\psi](\bx) & = \int_{\partial D} \frac{\partial (|\bx-\by|^{2j} (b_j \ln |\bx-\by| + c_j))}{\partial \bnu(\bx)} \psi(\by) \, ds(\by).
		\end{split}
	\end{equation}
	\begin{lem}\cite{AFGL007}\label{lem:asmsk}
		The following estimates hold in \(\mathcal{L}(L^2(\partial D), H^1(\partial D))\) and \(\mathcal{L}(L^2(\partial D), L^2(\partial D))\), respectively : 
		\[
		\begin{split}
			S_D^k &= \hat{S}_D^k + k^2 \ln k \, S_{D,1}^{(1)} + k^2 S_{D,1}^{(2)} + O(k^4 \ln k),\\
			K_D^{k,*} &= K_D + k^2 \ln k \, K_{D,1}^{(1)} + k^2 K_{D,1}^{(2)} + O(k^4 \ln k).
		\end{split}
		\]
	\end{lem}

	Based on the asymptotic expansions of the single layer potential operator and the N-P operator $S_D^k$ and $K_D^{k,*}$ in Lemma \ref{lem:asmsk}, we have the asymptotic expansion of the operator $\Acal(\omega,\delta)$.

	\begin{lem}\label{lem:asoa}
		For the operator $\Acal(\omega,\delta)$ from \( \Lcal(L^2(\partial D) \times L^2(\partial D)\) to \( H^1(\partial D) \times L^2(\partial D)) \) defined in \eqref{eq:or2}, we have
		\[
		\Acal(\omega, \delta) = \Acal_0 + \omega^2 \ln \omega \, \Acal_{1,1,0} + \omega^2 \Acal_{1,2,0} + \delta \Acal_{0,1} + \Ocal(\delta \omega^2 \ln \omega) + \Ocal(\omega^4 \ln \omega),
		\]
		where
		\[
		\Acal_0 = \begin{pmatrix}
			\hat{S}_D^{k_b} & -\hat{S}_D^{k} \\
			- \frac{1}{2} I + K_D^* & 0
		\end{pmatrix}, \quad \Acal_{1,1,0} = \begin{pmatrix}
			v_b^{-2} S_{D,1}^{(1)} & -v^{-2} S_{D,1}^{(1)} \\
			v_b^{-2} K_{D,1}^{(1)} & 0
		\end{pmatrix},
		\]
		\[
		\Acal_{1,2,0} = \begin{pmatrix}
			v_b^{-2} (-\ln v_b \, S_{D,1}^{(1)} + S_{D,1}^{(2)}) & -v^{-2} (-\ln v \, S_{D,1}^{(1)} + S_{D,1}^{(2)}) \\
			v_b^{-2} (-\ln v_b \, K_{D,1}^{(1)} + K_{D,1}^{(2)}) & 0
		\end{pmatrix},
		\]
		and
		\[
		\Acal_{0,1} = \begin{pmatrix}
			0 & 0 \\
			0 & - \left( \frac{1}{2} I + K_D^* \right)
		\end{pmatrix}.
		\]
	\end{lem}
	\begin{lem}\cite{ak162}\label{lem:kse0}
		For some $\varphi\in L^2(\partial D)$ with $\int_{\partial D} \varphi =0$, if it holds that $S_D[\varphi](\bx)=0$ for all $\bx\in\partial D$, then $\varphi=0$.
	\end{lem}
	
	If functions $\varphi\in L^2(\partial D)$ do not satisfy $\int_{\partial D} \varphi =0$, i.e., $\int_{\partial D} \varphi \neq 0$, the operator $S_D$ may have a nontrivial kernel \cite{ak162}.  Moreover, we can show that the dimension of the $\ker S_D$ is at most $1$.
	\begin{lem}
		For the single-layer potential operator $S_D$ from $L^2(\partial D)$ to $H^1(\partial D)$, there holds that 
		\[
		\dim \ker S_D \leq 1.
		\]
	\end{lem}
	\begin{proof}
		Assume that two functions $\varphi, \psi\in \ker S_D$. Then we define a new function 
		\[
		\phi = \frac{\varphi}{\int_{\partial D} \varphi} -  \frac{\psi}{\int_{\partial D} \psi}.
		\]
		It is obvious that $S_D[\phi]=0$ and $\int_{\partial D} \phi=0$. From Lemma \ref{lem:kse0}
		, we can obtain $\phi=0$. The proof is completed.    
	\end{proof}
	From the discussion above, the operator $S_D$ may not be invertible. Nevertheless, we then show that the operator $\hat{S}_D^k$ is always invertible. 
	\begin{lem}\label{lem:invsk}
		For any $k\in \mathbb{C}$ with $\Im\ln k\neq \pi/2$,  the operator $\hat{S}_D^k$ defined in \eqref{def:SD} is invertible from $L^2(\partial D)$ to $H^1(\partial D)$. 
	\end{lem}
	\begin{proof}
		Since the operator $\hat{S}_D^k$ is a Fredholm operator \cite{ak162}, we just need to consider the kernel of the operator $\hat{S}_D^k$. Assume that the function $\varphi\in L^2(\partial D)$  satisfies
		\begin{equation}\label{eq:has}
			\hat{S}_D^k[\varphi] = S_D [\varphi](\bx) + \eta_k \int_{\partial D} \varphi(\by) \, ds(\by) =0. 
		\end{equation}
		From the discussion below Lemma \ref{lem:kse0}, the operator $S_D$ may have a nontrivial kernel. Thus we first consider the case that there exists a nontrivial $\varphi_0$ with $\int_{\partial D} \varphi_0=1$ such that $S_D[\varphi_0]=0$. Since $S_D$ is self-adjoint, from \eqref{eq:has} we obtain that 
		\[
		\eta_k \left(\int_{\partial D} \varphi \right) \left(\int_{\partial D} \varphi_0 \right)=0.
		\]
		Thus, we derive $\int_{\partial D} \varphi=0$ and $S_D [\varphi]=0$. Then by Lemma \ref{lem:kse0} we can conclude $\varphi=0$. 
		
		Next, we consider the case that the operator $S_D$ is invertible. Then the equation \eqref{eq:has} gives that $S_D [\varphi](\bx) = -\eta_k \int_{\partial D} \varphi=c_1$ for some constant $c_1$. Define $\varphi_0=S_D^{-1}[1]$ and it is obvious that $\varphi_0$ is a real function from the definition of the operator  $S_D$ in \eqref{def:SD}. Then we have $\varphi=c_1 \varphi_0$. Substituting the expression of $\varphi$ into \eqref{eq:has} yields that 
		\[
		c_1\left( 1 + \eta_k \int_{\partial D} \varphi_0 \right)=0.
		\]
		By the choice of the parameter $k$ , $\eta_k$ has non-zero imaginary part. Thus we can obtain $c_1=0$. Finally, we conclude that the operator $\hat{S}_D^k$ is invertible. The proof is completed.
	\end{proof}
	Moreover, the following lemma holds for the operator $S_D$.
	{
		\begin{lem}\cite{ak162}\label{lem:invscon}
			The operator $\Tcal: L^2(\partial D) \times \mathbb{C} \rightarrow H^1(\partial D) \times \mathbb{C}$ defined by
			\[
			\Tcal[\varphi, c] = \left( S_D[\varphi] + c, \int_{\partial D} \varphi \right)
			\]
			has a bounded inverse.
		\end{lem}
	}

	\begin{lem}\cite{ak162}\label{lem:innp}
		For any $\varphi \in L^2(\partial D)$, it holds that 
		$$ \int_{\partial D}\left(-\frac{1}{2}I+K_D^*\right)[\varphi] \cdot \chi_{\partial D_i}\ \mathrm{d}s =0 \quad\mbox{for} \quad i=1,2,$$
		and 
		$$\int_{\partial D}\left(\frac{1}{2}I+K_D^*\right)[\varphi]\cdot \chi _{\partial D_i}\ \mathrm{d}s=\int_{\partial D_i} \varphi\ \mathrm{d}s \quad\mbox{for} \quad i=1,2,$$
		where $\chi$ is the characteristic function. 
	\end{lem}
	
	\begin{lem}\label{lem:ink12}
		For any $\varphi \in L^2(\partial D)$, we have that 
		\[
		\int_{\partial D}K_{D,1}^{(1)}[\varphi]\cdot \chi _{\partial D_i}= 4b_1|D_i|\int_{\partial D}\varphi.
		\]
	\end{lem}
	\begin{proof}
		It follows from \eqref{def-KDj} that
		\begin{align*}
			\int_{\partial D}K_{D,1}^{(1)}[\varphi]\cdot \chi _{\partial D_i}&=b_1\int_{\partial D}\int_{\partial D}\partial_{\bnu_\bx}|\bx-\by|^2\cdot \varphi(\by)\ ds(\by) \chi_{\partial D_i}(\bx)\ ds(\bx)\\
			&=b_1\int_{\partial D}\varphi(\by)\int_{\partial D}\partial_{\bnu_\bx}|\bx-\by|^2 \chi_{\partial D_i}(\bx)\ ds(\bx)\ ds(\by)\\
			&=b_1\int_{\partial D}\varphi(\by)\int_{D_i}\Delta_\bx |\bx-\by|^2\ d\bx\ ds(\by)\\
			&=4b_1|D_i|\int_{\partial D}\varphi(\by)\ ds(\by).
		\end{align*}
		The proof is completed.
	\end{proof}
	
	{
		\begin{lem}\label{lem:ink22}
			For any $\varphi \in L^2(\partial D)$, we have that 
			\[
			\int_{\partial D} K_{D,1}^{(2)}[\varphi]\cdot   \chi_{\partial D_i}\  =-4 b_1 \ln k |D_i|\int_{\partial D}\varphi(\by)\ ds(\by)-\int_{D_i}\hat{S}^k_D[\varphi](\bx)\ d\bx.
			\]
		\end{lem}
		
		\begin{proof}
			From \eqref{def-KDj}, we have that 
			\begin{align*}
				&  \int_{\partial D} K_{D,1}^{(2)}[\varphi](\bx)\cdot   \chi_{\partial D_i}(\bx)\ \mathrm{d}s(\bx)\\
				&=\int_{\partial D}\varphi(\by) \int_D \Delta (b_1|\bx-\by|^2\ln |\bx-\by|+c_1|\bx-\by|^2)\chi_{\partial D_i}(\bx)\  ds(\bx) ds(\by)\\
				&=\int_{\partial D} \int_{D_i}(4(b_1+c_1)+4b_1\ln |\bx-\by|)\ d\bx \ \varphi(\by)\ ds(\by)\\
				&=-4 b_1 \ln k |D_i|\int_{\partial D}\varphi(\by)\ ds(\by)-\int_{D_i}\hat{S}_D^k[\varphi](\bx)\ d\bx.
			\end{align*}
			This completes the proof.
		\end{proof}
		
		
		For further analysis, we introduce the following functions
		\begin{equation}\label{eq:insk}
			\begin{aligned}
				\xi _1&=(\hat{S}_D^{k_b})^{-1}[\chi_{\partial D_1}],\qquad \xi _2=(\hat{S}_D^{k_b})^{-1}[\chi_{\partial D_2}],\\
				\zeta_1&=(\hat{S}_D^{k})^{-1}[\chi_{\partial D_1}],\qquad \zeta_2=(\hat{S}_D^{k})^{-1}[\chi_{\partial D_2}].
			\end{aligned}
		\end{equation}
		They are all located in the the kernel of the operator $-\frac{1}{2}I+K_D^{*}$ (cf.\cite{AD0049}).
	}

	\section{The analysis of the resonances}\label{sec-resonance}
	
	In this section, we consider the resonant phenomena for the system \eqref{eq:or1}. To obtain explicit expressions of resonant frequencies, we assume that the domains $D_1$ and $D_2$ are symmetric with respect to the $\bx_1$ axis in the sense that 
	\begin{equation}\label{eq:smd}
		D_2=\{ (\bx_1, -\bx_2): \bx\in D_1 \}.
	\end{equation}
	Denote by
	\begin{equation}\label{eq:deaij}
		\alpha_{ij} = \int_{\partial D_j}\zeta_i,\quad i,j=1,2.
	\end{equation}
	
	\begin{lem}\label{lem:ralij}
		Let the domain $D$ be symmetric in the sense of \eqref{eq:smd}. For the functions $\zeta_1$ and $\zeta_2$ defined in \eqref{eq:insk}, we have that 
		\[
		\zeta_2(\by_1,\by_2)= \zeta_1(\by_1, -\by_2) \quad\mbox{and}\quad \zeta_1(\by_1,\by_2)= \zeta_2(\by_1, -\by_2). 
		\]
		For the parameters $\alpha_{ij}$ given in \eqref{eq:deaij}, it holds that 
		\[
		\alpha_{21} =\alpha_{12} \quad\mbox{and}\quad \alpha_{11} = \alpha_{22}.
		\]
	\end{lem}
	\begin{proof}
		From the definition of $\zeta_1$ in \eqref{eq:insk} and using \eqref{def:SD}, we have that
		\begin{align*}
			\chi _{\partial D_1}=\hat{S}_D^k[\zeta_1]=&\int_{\partial D}\ln |\bx-\by|\zeta_1(\by)\ \mathrm{d}s(\by)+\eta_k \int_{\partial D}\zeta_1(\by)\ ds(\by) \\
			=&\int_{\partial D_1}\ln |\bx-\by|\zeta_1(\by)\ ds(\by)+\eta_k \int_{\partial D_1}\zeta_1(\by)\ ds(\by) \\
			&+\int_{\partial D_2}\ln |\bx-\by|\zeta_1(\by)\ ds(\by)+\eta_k \int_{\partial D_2}\zeta_1(\by)\ ds(\by).
		\end{align*}
		A direct calculation shows that 
		\begin{align*}
			&\int_{\partial D_1}\ln \sqrt{(\bx_1-\by_1)^2+(\bx_2-\by_2)^2}\zeta_1(\by_1,\by_2)ds(\by)+\eta_k \int_{\partial D_1} \zeta_1(\by_1,\by_2)ds(\by)\\
			&=\int_{\partial D_2}\ln \sqrt{(\bx_1-\by_1)^2+(\bx_2-(-\by_2))^2}\zeta_1(\by_1,-\by_2)ds(\by) +\eta_k \int_{\partial D_2}\zeta_1(\by_1,-\by_2)ds(\by)\\
			&=\int_{\partial D_2}\ln|\widetilde{\bx}-\by|\zeta_1(\by_1,-\by_2)ds(\by)+\eta\int_{\partial D_2}\zeta_1(\by_1,-\by_2)ds(\by),
		\end{align*}
		where $\widetilde{\bx}=(\bx_1, -\bx_2)$. Following the same discussion, we have that 
		\[
		\begin{split}
			&\int_{\partial D_2}\ln |\bx-\by|\zeta_1(\by)\ ds(\by)+\eta_k \int_{\partial D_2}\zeta_1(\by)\ ds(\by) \\
			& =\int_{\partial D_1}\ln |\widetilde{\bx}-\by|\zeta_1(\by_1,-\by_2)\ ds(\by)+\eta_k \int_{\partial D_1}\zeta_1(\by_1,-\by_2)\ ds(\by).
		\end{split}
		\]
		From the last three identities, we obtain that 
		\[
		\hat{S}_D^k[\zeta_1(\by_1, -\by_2)] = \chi_{\partial D_2}.
		\]
		Thus, we conclude that $\zeta_2(\by_1,\by_2)= \zeta_1(\by_1, -\by_2)$ as $\hat{S}_D^k$ is invertible. Similarly, we prove $\zeta_1(\by_1,\by_2)= \zeta_2(\by_1, -\by_2)$.
		
		Moreover, we have 
		\begin{align*}
			\alpha_{21}=\int_{\partial D_1}\zeta _2(\by_1,\by_2)=\int_{\partial D_2}\zeta_2(\by_1,-\by_2)=\int_{\partial D_2}\zeta_1(\by_1,\by_2)=\alpha_{12},\\
			\alpha_{22}=\int_{\partial D_2}\zeta_2(\by_1,\by_2)=\int_{\partial D_1}\zeta_2(\by_1,-\by_2)=\int_{\partial D_1}\zeta_1(\by_1,\by_2)=\alpha_{11}.
		\end{align*}
		The proof is completed. 
	\end{proof}
	
	\begin{rem}\label{rem:a12}
		From Lemma \ref{lem:ralij}, if the domain $D$ is symmetric, we further have that 
		\begin{equation}\label{eq:al12}
			\alpha_1 = \alpha_2,
		\end{equation}
		where 
		\begin{equation*}
			\alpha_1 =\int_{\partial D}\zeta_1=\alpha_{11}+\alpha_{12}, \quad \quad \alpha_2=\int_{\partial D}\zeta_2=\alpha_{21}+\alpha_{22}.
		\end{equation*}
	\end{rem}

	\begin{rem}\label{rem:apo}
		As in Section \ref{sec:blowup} below, we take a large constant $R$ such that $\overline{D_1\cup D_2}\subset B_R$. From \eqref{eq:insk}, we know that $w_i:=\hat{S}_D^{k}[\zeta_i]$ is the weak solution of 
		\begin{align*}
			\begin{cases}
				\Delta w_i=0\quad &\text{in} \quad \mathbb{R}^2\setminus\overline{D},\\
				w_i=\delta_{ij}\quad &\text{on} \quad \partial D_j,~i,j=1,2,\\
				w_i=\hat{S}_D^{k}[\zeta_i] \quad &\text{on} \quad \partial B_R.\\
			\end{cases}
		\end{align*}
		Then we have 
		\begin{align}\label{sol-wi}
			\begin{cases}
				\Delta (w_2-w_1)=0\quad &\text{in} \quad \mathbb{R}^2\setminus\overline{D},\\
				w_2-w_1=-1\quad &\text{on} \quad \partial D_1,\\
				w_2-w_1=1\quad &\text{on} \quad \partial D_2,\\
				w_2-w_1=\Ocal(|\bx|^{-1}) \quad &\text{as} \quad |\bx|\rightarrow\infty.\\
			\end{cases}
		\end{align}
		The last condition in \eqref{sol-wi} follows from the definition of the operator $\hat{S}_D^{k}$ in \eqref{def:SD}, Remark \ref{rem:a12} and the asymptotic expansion of the function $\ln{\abs{\bx-\by}} = \ln{\abs{\bx}} + \Ocal(\abs{\bx}^{-1})$ with $|\bx| \gg 1$ and $\by\in \partial D$. Following the same reason, we have that $\partial_{\bnu}(w_2-w_1)=\Ocal(|\bx|^{-2})$.
		Multiplying the equation in \eqref{sol-wi} by $w_2-w_1$ and integrating by parts in $B_R\setminus\overline{D}$, we obtain 
		\begin{equation*}
			\int_{\partial D_1}\partial_{\bnu}(w_2-w_1)-\int_{\partial D_2}\partial_{\bnu}(w_2-w_1)+\int_{\partial B_R}(w_2-w_1)\partial_{\bnu}(w_2-w_1)=\int_{B_R\setminus\overline{D}}|\nabla(w_2-w_1)|^2.
		\end{equation*}
		Since $w_2-w_1=\Ocal(|\bx|^{-1})$ and $\partial_{\bnu}(w_2-w_1)=\Ocal(|\bx|^{-2})$, as $R\rightarrow+\infty$, we have 
		\begin{equation}\label{alpha12-alpha11}
			\int_{\partial D_1}\partial_{\bnu}(w_2-w_1)-\int_{\partial D_2}\partial_{\bnu}(w_2-w_1)=\int_{\mathbb{R}^2\setminus\overline{D}}|\nabla(w_2-w_1)|^2.
		\end{equation}
		By using the fact that $\{\zeta_1,\zeta_2\}$ belong to the kernel of the operator $\left(-\frac{1}{2}I+K_D^*\right)$ and the jump condition \eqref{eq:jump}, we have 
		\begin{equation}\label{form-alphaij}
			\int_{\partial D_j}\partial_{\bnu}w_i|_{+}\ ds=\int_{\partial D_j}\partial_{\bnu}\hat S_D^{k}[\zeta_i]|_{+}\ ds=\int_{\partial D_j}\zeta_i\ ds=\alpha_{ij},\quad i,j=1,2.
		\end{equation}
		This together with  $\alpha_{21}=\alpha_{12}$ and $\alpha_{11}=\alpha_{22}$ yields
		\begin{equation*}
			\alpha_{12}-\alpha_{11}=\int_{\partial D_1}\partial_{\bnu}(w_2-w_1)=\int_{\partial D_2}\partial_{\bnu}(w_1-w_2).
		\end{equation*}
		Thus, combining with \eqref{alpha12-alpha11}, we derive $\alpha_{12}>\alpha_{11}$.
	\end{rem}
	
	{
		
		The following lemma provides a more detailed estimate of $\alpha_1$.
		\begin{lem}\label{lem:zeta1}
			Let $C$ denote the norm of the operator $\Tcal$ defined in Lemma \ref{lem:invscon}, and 
			\[
			\tau = \frac{1}{2\pi} ( \gamma - \ln 2) - \frac{\ri}{4}.
			\]
			If $k\ll 1$ is chosen such that $\abs{2\pi(C + \abs{\tau})/\ln k} <1$, then 
			the parameter $\alpha_1 =\int_{\partial D}\zeta_1=\alpha_{11}+\alpha_{12}$ has the following estimate:
			\begin{equation}\label{eq:zeta1-est}
				\alpha_1 = \frac{\pi}{\ln k}(1 + o(1)).
			\end{equation}
		\end{lem}
		\begin{proof}
			From $\zeta_1 = (\hat{S}_D^{k})^{-1}[\chi_{\partial D_1}]$ and the definition of $\hat{S}_D^{k}$ in \eqref{def:SD}, we have that
			\begin{equation}\label{eq:zeta1}
				\hat{S}_D^{k}[\zeta_1] = \frac{\ln k}{2\pi}\int_{\partial D}\zeta_1 + \tau\int_{\partial D}\zeta_1 + S_D[\zeta_1] = \chi_{\partial D_1}.
			\end{equation}
			Next, we prove that the function $\zeta_1$ has the following asymptotic expansion:
			\begin{equation}\label{eq:zeta1-exp}
				\zeta_1(\by) = \zeta_{1,0}(\by) + \sum_{j=1}^{\infty}\frac{1}{(\ln k)^j}\zeta_{1,j}(\by).
			\end{equation}
			Substituting the last equation into \eqref{eq:zeta1} and comparing the order of the parameter $\ln k$, we have that
			\begin{equation}\label{eq:zeta1e1}
				\int_{\partial D}\zeta_{1,0} = 0,  \quad \frac{1}{2\pi}\int_{\partial D}\zeta_{1,1} + S_D[\zeta_{1,0}] = \chi_{\partial D_1},
			\end{equation}
			and for $j\geq 1$,
			\begin{equation}\label{eq:zeta1e2}
				\begin{cases}
					\frac{1}{2\pi}\int_{\partial D}\zeta_{1,j+1} + \tau \int_{\partial D}\zeta_{1,j} + S_D[\zeta_{1,j}] = 0,  \\ 
					\int_{\partial D}\zeta_{1,j} = t_j, 
				\end{cases}
			\end{equation}
			where $t_j$ are constants.
			
			The existence of the solution $\zeta_{1,0}$ in \eqref{eq:zeta1e1} is guaranteed by Lemma \ref{lem:invscon}.
			Next, we provide the detailed estimate of the parameter $t_1=\int_{\partial D}\zeta_{1,1}$. By Lemma \ref{lem:ralij}, we have that similar to \eqref{eq:zeta1e1}
			\begin{equation}\label{eq:zeta1e3}
				\int_{\partial D}\zeta_{2,0} = 0,  \quad \frac{1}{2\pi}t_1 + S_D[\zeta_{2,0}] = \chi_{\partial D_2}.
			\end{equation}
			Adding \eqref{eq:zeta1e3} and \eqref{eq:zeta1e1} gives that
			\begin{equation}\label{eq:zeta1e4}
				\int_{\partial D}\zeta_{1,0} + \zeta_{2,0} = 0,  \quad   S_D[\zeta_{1,0} + \zeta_{2,0}] + \frac{1}{\pi}t_1-1 = 0.
			\end{equation}
			By Lemma \ref{lem:invscon}, we conclude $t_1=\pi$ since the operator $\mathcal{T}$ is invertible.
			
			Next, we consider the system \eqref{eq:zeta1e2}. By Lemma \ref{lem:invscon}, the solution $\zeta_{1,1}$ exists and satisfies the following estimate
			\[
			\|\zeta_{1,1}\|_{L^2(\partial D)} + \abs{ \frac{1}{2\pi}\int_{\partial D}\zeta_{1,2} + \tau t_1 }  \leq C |t_1|,
			\]
			where $C$ is the norm of the operator $\Tcal$ defined in Lemma \ref{lem:invscon}.
			From the last inequality and the triangle inequality, we have that
			\[
			\|\zeta_{1,1}\|_{L^2(\partial D)} \leq C|t_1| \quad \mbox{and} \quad \abs{\int_{\partial D}\zeta_{1,2}} \leq 2\pi(C + \abs{\tau}) |t_1|.
			\]
			By the same argument, we can show that
			\[
			\|\zeta_{1,j}\|_{L^2(\partial D)} \leq C \left( 2\pi(C + \abs{\tau}) \right)^{j-1}|t_1| \quad \mbox{and} \quad \abs{\int_{\partial D}\zeta_{1,j+1}} \leq \left( 2\pi(C + \abs{\tau}) \right)^{j} |t_1|.
			\]
			Thus, if $2\pi(C + \abs{\tau})/\ln k <1$, we conclude that the asymptotic expansion \eqref{eq:zeta1-exp} converges with respect to the norm $\|\cdot\|_{L^2(\partial D)}$.
			
			Finally, from the asymptotic expansion \eqref{eq:zeta1-exp} and $t_1=\pi$, we obtain \eqref{eq:zeta1-est}.
			The proof is completed.
		\end{proof}
	}

	Next, we study the kernel of the operator $\Acal(\omega ,\delta )$. To that end, we first analyze the kernel of the operator $\Acal_0$, which is the leading order of the operator $\Acal(\omega ,\delta )$. 
	\begin{lem}\label{lem:kea0}
		The kernel of the operator $\Acal_0$ defined in Lemma \ref{lem:asoa} is nontrivial.
	\end{lem}
	\begin{proof}
		From the definitions of the functions $\xi_1, \xi_2$ and $\zeta_1, \zeta_2$ in \eqref{eq:insk} and the fact that $\xi_1$ and $\xi_2$ given in \eqref{eq:insk} are in the kernel of the operator $\left(-\frac{1}{2}I+K_D^*\right)$, we see that
		\[
		\left( \xi_1, \;\zeta_1 \right) \quad \mbox{and} \quad \left( \xi_2, \; \zeta_2 \right) 
		\]
		are in the kernel of the operator $\Acal_0$.
	\end{proof}
	
	Since the kernel of the operator $\Acal_0$ is nontrivial from Lemma \ref{lem:kea0}, by the Gohberg–Sigal theory \cite{LZ2861}, there exist frequencies $\omega\in\mathbb{C}$ such that the operator $\Acal(\omega,\delta)$ in \eqref{eq:or2} has nontrivial kernels. In other words, the system \eqref{eq:or1} has resonant frequencies. The explicit resonant frequencies are concluded in the following theorem.
	
	\begin{thm}\label{thm:refre}
		There are two resonant frequencies for the system \eqref{eq:or1}. One resonant frequency is
		\[
		\omega_1 = \omega_{1,1}(1 +o(1)),
		\]
		where the leading order $\omega_{1,1}$ satisfies
		\[
		\frac{\omega_{1,1}^2}{v_b^2}\ln \left(\frac{\omega_{1,1}}{v}\right) {|D_1|} + \pi \delta =0.
		\]
		 The corresponding resonant eigenfunctions satisfy 
		\begin{equation}\label{eq:reeifu1}
			\varphi=\xi_1+\xi_2+\Ocal(k_b^2\ln k_b), \quad \psi =\zeta_1+\zeta_2+\Ocal(k^2\ln k).
		\end{equation}
		The other resonant frequency is 
		\begin{equation*}
			\omega_2= v_b\sqrt{\frac{\alpha_{12} - \alpha_{11}}{|D_1|}} \sqrt{\delta} (1 + o(1)).
		\end{equation*}
		The corresponding resonant eigenfunctions satisfy 
		\begin{equation}\label{eq:reeifu2}
			\varphi=\xi_1-\xi_2+\Ocal(k_b^2\ln k_b), \quad \psi =\zeta_1-\zeta_2+\Ocal(k^2\ln k).
		\end{equation}
	\end{thm}
	
	\begin{proof}
		From the discussion below Lemma \ref{lem:kea0}, there exists frequencies $\omega\in\mathbb{C}$ such that the operator $\Acal(\omega,\delta)$ in \eqref{eq:or2} has a nontrivial kernel, namely
		$$
		\Acal(\omega,\delta) \begin{bmatrix}
			\varphi\\
			\psi 
		\end{bmatrix}=0.
		$$
		From the expression of the operator $\Acal(\omega,\delta)$, the equation above is equivalent to 
		$$
		\begin{cases}
			S_D^{k_b}[\varphi]-S_D^k[\psi]=0\\
			(-\frac{1}{2}I+K_D^{k_b,*})[\varphi]-\delta (\frac{1}{2}I+K_D^{k,*})[\psi]=0.
		\end{cases}
		$$
		From the asymptotic expansions in \eqref{eq:assin} and \eqref{eq:asnp} for $k\ll 1$, the last equation is equivalent to
		\begin{equation}\label{eq:exp1}
			\hat{S}_D^{k_b}[\varphi] - \hat{S}_D^k[\psi] = \Ocal(k^2\ln k)
		\end{equation}
		and
		\begin{equation}\label{eq:exp2}
			\left(-\frac{1}{2}I + K_D^* + k_b^2 \ln k_b K_{D,1}^{(1)} + k_b^2 K_{D,1}^{(2)} \right)[\varphi] - \delta \left(\frac{1}{2}I + K_D^*\right)[\psi] = \Ocal(k_b^4 \ln k_b + k^2 \delta \ln k).
		\end{equation}
		From \eqref{eq:exp2} and $\delta\ll 1$, we have that 
		\begin{equation*}
			\varphi-\Ocal(k_b^2\ln k_b + \delta) \in \ker \left(-\frac{1}{2}I+K^*_D\right).
		\end{equation*}
		Thus, it follows from \eqref{eq:insk} that 
		\begin{equation}\label{eq3}
			\varphi =d_1\xi_1+d_2\xi_2+\Ocal(k_b^2\ln k_b + \delta),
		\end{equation}
		where $d_1$ and $d_2$ are two constants. Then the  equation \eqref{eq:exp1}  implies that 
		\begin{equation}\label{eq4}
			\psi =d_1\zeta_1+d_2\zeta_2+\Ocal(k^2\ln k + \delta).
		\end{equation}
		{
			Substituting \eqref{eq3} and \eqref{eq4} into the equation \eqref{eq:exp2} gives that
			\[
			\begin{split}
				\left(-\frac{1}{2}I+K^*_D\right)[\varphi]+ & k_b^2\ln k_b K_{D,1}^{(1)}[d_1\xi_1+d_2\xi_2]+k_b^2 K_{D,2}^{(2)}[d_1\xi_1+d_2\xi_2]\\
				& -  \delta \left(\frac{1}{2}I+K_D^*\right)[d_1\zeta_1+d_2\zeta_2]=\Ocal\left(k_b^4 (\ln k_b)^2 +k^2\delta \ln k + \delta^2\right).
			\end{split}
			\]
			Multiplying both sides of the last equation by $\chi_{\partial D_i}$, integrating on $\partial D$ together with the help of Lemmas \ref{lem:innp}--\ref{lem:ink22}, we have that 
			\begin{align*}
				-k_b^2|D_i|(d_1\delta_{1i}+d_2\delta_{2i})  - \delta \left(d_1\int_{\partial D_i}\zeta_1+d_2\int_{\partial D_i}\zeta_2\right) =\Ocal\left(k_b^4 (\ln k_b)^2 +k^2\delta \ln k + \delta^2\right),
			\end{align*}
			where $\delta_{ij}$ is the Kronecker delta symbol.
			Since the domain $D$ is symmetric, from Lemma \ref{lem:ralij} we have that for $i=1,2$,
			\begin{equation}\label{eq:eqk}
				\begin{split}
					\left(  k_b^2|D_1|I  + \delta B_2 \right)\begin{bmatrix}
						d_1\\
						d_2
					\end{bmatrix} 
					=\Ocal\left(k_b^4 (\ln k_b)^2 +k^2\delta \ln k + \delta^2\right),
				\end{split}
			\end{equation}
			where $I$ is the identity matrix and $B_2$ is given by
			\[
			B_2=\begin{pmatrix}
				\alpha_{11} & \alpha_{12}\\
				\alpha_{12} & \alpha_{11}
			\end{pmatrix}.
			\]
			Direct calculation shows that the eigensystem of $B_2$ is given by
			\begin{equation}\label{eq:eib12}
				B_2 T = T \Lambda_2,
			\end{equation}
			where
			\[
			T=\begin{pmatrix}
				\frac{1}{\sqrt{2}} &-\frac{1}{\sqrt{2}}\\
				\frac{1}{\sqrt{2}} & \frac{1}{\sqrt{2}}
			\end{pmatrix}, 
			\quad
			\Lambda_2=\begin{pmatrix}
				\alpha_{11} + \alpha_{12} & 0\\
				0 & \alpha_{11} - \alpha_{12}
			\end{pmatrix}.
			\]
			Finally, from the equations \eqref{eq:eqk} and \eqref{eq:eib12} and Lemma \ref{lem:zeta1}, we obtain that the leading order of one of the resonant frequencies satisfies
			\begin{equation}\label{eq:lnom}
				\frac{\omega_{1,1}^2}{v_b^2}\ln \left(\frac{\omega_{1,1}}{v}\right) {|D_1|} + \pi \delta =0,
			\end{equation}
			due to $k_b=\omega/v_b$ and $k=\omega/v$ from \eqref{eq:aux-param}.   }
		From \eqref{eq3} and \eqref{eq4}, and the eigenvectors of $B_2$ in \eqref{eq:eib12}, we conclude that the corresponding eigenfunctions satisfy
		\[
		\varphi=\xi_1+\xi_2+\Ocal(k_b^2\ln k_b), \quad \psi =\zeta_1+\zeta_2+\Ocal(k^2\ln k).
		\]
		{Here, we have used the fact $\delta=\Ocal(k_b^2\ln k_b)$ from \eqref{eq:lnom}.}
		Further, the second resonant frequency satisfies 
		\begin{equation}\label{eq:sqom}
			\omega= v_b\sqrt{\frac{\alpha_{12} - \alpha_{11}}{|D_1|}} \sqrt{\delta} (1 + o(1)).
		\end{equation}
		The corresponding eigenfunctions satisfy
		\[
		\varphi=\xi_1-\xi_2+\Ocal(k_b^2\ln k_b), \quad \psi =\zeta_1-\zeta_2+\Ocal(k^2\ln k),
		\]
		{where we use the fact $\delta=\Ocal(k_b^2)$ from \eqref{eq:sqom}.}
		The proof is completed.
		\end{proof}

	\begin{rem}\label{rem:diff3d}
		The derivation of the 2D resonant frequencies is more involved than in the 3D case. For the 3D Helmholtz system, both resonant frequencies are obtained by just expanding the N-P operator to the first order. However, for the 2D Helmholtz system, the first resonant frequency is obtained by expanding the N-P operator to the order $\Ocal(k^2 \ln k)$. The other resonant frequency is derived by expanding the N-P operator to the next order $\Ocal(k^2)$. That is why the two resonant frequencies in 2D have different leading orders. 
	\end{rem}
	
	{
		As shown in Theorem \ref{thm:refre}, one resonant frequency satisfies 
		\[
		\omega= v_b\sqrt{\frac{\alpha_{12} - \alpha_{11}}{|D_1|}} \sqrt{\delta} (1 + o(1)).
		\]
		In Remark \ref{rem:apo}, we have shown that $\alpha_{12} - \alpha_{11}$ is positive. Moreover, as shown in Appendix \ref{Appendix}, when the distance $\varepsilon$ between the two resonators tends to zero, we can prove the following result.
		\begin{cor}
			The term $\alpha_{12}-\alpha_{11}$ appeared in the last formula has the following estimate
			\begin{align*}
				\alpha_{12}-\alpha_{11}=\frac{2\pi}{\sqrt{\lambda\varepsilon}}+\varepsilon^{\frac{\alpha-1}{2}}\Ocal(1),
			\end{align*}
			where $\alpha\in(0,1)$ and $\lambda$ is the curvature of $\partial D$ at $(0,\varepsilon/2)$ and $(0,-\varepsilon/2)$. Thus, the corresponding resonant frequency is of the order $\Ocal(\delta^{1/2}/\varepsilon^{1/4})$. If the parameter $\varepsilon$ is chosen $\varepsilon = \Ocal(\delta^{\beta})$ with $0<\beta<2$, the resonant frequency is of the order $\Ocal(\delta^{(1-\beta/2)/2})$.
		\end{cor}
	}

	\section{Blow-up analysis}\label{sec:blowup}
	In this section, we analyze the blow-up behavior of the resonant modes as the two resonators get closer to each other. For this, we first take a large constant $R$  such that $\overline{D_1\cup D_2}\subset B_R$. {In what follows, we denote $\Omega := B_R \setminus \overline{D}$ for simplicity and let $\Omega_r$ (defined in \eqref{narrowreg}) represent the narrow region between $D_1$ and $D_2$.} 
	
	From Theorem \ref{thm:refre}, the first resonant mode in the region $\mathbb{R}^2\setminus\overline{D}$ satisfies
	\begin{equation}\label{eq:blowup1}
		\begin{cases}
			\Delta u_1+k^2u_1 =0 \quad &\text{in} \quad \mathbb{R}^2\setminus\overline{D},\\
			u_1=1+g\quad &\text{on} \quad \partial D_1,\\
			u_1=1+h\quad &\text{on} \quad \partial D_2,\\
			u_1=f_1 \quad &\text{on} \quad \partial B_R,\\
		\end{cases}	
	\end{equation}
	where $g(\bx)=\Ocal(k^2\ln k)$, $h(\bx)=\Ocal(k^2\ln k)$, and
	\[
	f_1(\bx) = S_D^{k}[\psi](\bx), \quad \bx\in\partial B_R
	\]
	with $\psi$ defined in \eqref{eq:reeifu1}.
	
	For the second resonant mode, the corresponding resonant mode satisfies 
	\begin{equation}\label{eq:blowup}
		\begin{cases}
			\Delta u_2+k^2u_2=0\quad &\text{in} \quad \mathbb{R}^2\setminus\overline{D},\\
			u_2=1+o(1)\quad &\text{on} \quad \partial D_1,\\
			u_2=-1+o(1)\quad &\text{on} \quad \partial D_2,\\
			u_2=f_2\quad &\text{on} \quad \partial B_R, 
		\end{cases}
	\end{equation}
	where 
	\[
	f_2(\bx) = S_D^{k}[\psi](\bx), \quad \bx\in\partial B_R,
	\]
	with $\psi$ defined in \eqref{eq:reeifu2}.

	In order to prove the gradient estimates of solutions of \eqref{eq:blowup1} and \eqref{eq:blowup}, we first present more characteristics of the two domains $D_1$ and $D_2$. 
		By a translation and rotation of coordinates (if necessary), there exists a constant $R_0$ independent of $\varepsilon$, such that the sections of $\partial D_{1}$ and $\partial D_{2}$ near the origin, respectively, can be represented by
		\begin{align}\label{h1h2}
			x_{2}=\frac{\varepsilon}{2}+\mathcal{H}_{1}(x_1)\quad\mbox{and}\quad x_{2}=-\frac{\varepsilon}{2}+\mathcal{H}_{2}(x_1)\quad\mbox{for}~|x_1|<2R_0,
		\end{align}
		where $\varepsilon:=\mbox{dist}(D_1,D_2)$. In \eqref{h1h2}, the functions $\mathcal{H}_i\in C^{2,\alpha}(-2R_0,2R_0)$, $i=1,2$, have the expressions 
		\begin{align*}
			\mathcal{H}_{i}(x_1)=
			(-1)^{i+1}\frac{\lambda}{2}|x_1|^{2}+O(|x_1|^{2+\alpha}),
		\end{align*}
		and 
		\begin{equation*}
			\|\mathcal{H}_{i}\|_{C^{2,\alpha}(-2R_0,2R_0)}\leq C,
		\end{equation*}
		where $C$ is a positive constant independent of $\varepsilon$ and $\lambda$ is the curvature of $\partial D$ at $(0,\varepsilon/2)$ and $(0,-\varepsilon/2)$. Throughout the paper, the constant $C$ is independent of $\varepsilon$, and may vary from line to line in various inequalities. Here we would like to remark that our method can be applied to deal with the
		more general inclusions case, say, $\mathcal{H}_{i}(x_1)=
		(-1)^{i+1}\lambda_i|x_1|^{2}+O(|x_1|^{2+\alpha})$ with two positive constants $\lambda_1$ and $\lambda_2$.
		For $0<r\leq\,2R_0$, we define the narrow region between $\partial{D}_{1}$ and $\partial{D}_{2}$ as follows:
		\begin{equation}\label{narrowreg}
			\Omega_r:=\left\{(x_1,x_2)\in \mathbb{R}^{2}: -\frac{\varepsilon}{2}+\mathcal{H}_{2}(x_1)<x_2<\frac{\varepsilon}{2}+\mathcal{H}_{1}(x_1),~|x_1|<r\right\},
		\end{equation}
		and the vertical distance between $\partial{D}_{1}$ and $\partial{D}_{2}$ is denoted by
		\begin{equation}\label{delta_x'}
			\delta(x_1):=\varepsilon+\mathcal{H}_{1}(x_1)-\mathcal{H}_{2}(x_1)=\varepsilon+\lambda|x_1|^{2}+\Ocal(|x_1|^{2+\alpha}).
	\end{equation}
	The behavior of the gradients of the resonant modes as the two resonators are moved close together is stated as follows.
	
	\begin{thm}\label{thm:blowup1}
		Let $u_1$ be the solution of \eqref{eq:blowup1}. Then for $\bx\in\mathbb{R}^2\setminus\overline{D}$ and sufficiently small $\varepsilon>0$, we have 
		\begin{align*}
			|\nabla u_1(\bx)|&\leq \frac{C|g(x_1,\varepsilon/2+\mathcal{H}_{1}(x_1))-h(x_1,-\varepsilon/2+\mathcal{H}_{2}(x_1))|}{\varepsilon+\lambda|x_1|^2}\\
			&\quad+C\big(\|g\|_{C^2(\partial D_1)}+\|h\|_{C^2(\partial D_2)}+(k+1)\|u_{1}\|_{L^2(\mathbb{R}^2\setminus\overline{D})}+k+\|f_1\|_{L^{\infty}(\partial B_R)}\big),
		\end{align*}
		where $C>0$ is a constant independently of $\varepsilon$, $\mathcal{H}_{1}$ and $\mathcal{H}_{2}$ are defined in \eqref{h1h2} below, and $\lambda$ is the curvature of $\partial D$ at $(0,\varepsilon/2)$ and $(0,-\varepsilon/2)$.
	\end{thm}
	
	\begin{rem}
		As shown in Theorem \ref{thm:blowup1}, the blow-up behavior of $\nabla u_1$ depends on the boundary conditions $g$ and $h$, with no blow-up occurring when $g(x_1,\varepsilon/2+\mathcal{H}_{1}(x_1))\equiv h(x_1,-\varepsilon/2+\mathcal{H}_{2}(x_1))$. This is consistent with the result in \cite{JLX2019} concerning elliptic systems in narrow regions.
	\end{rem}
	
	\begin{thm}\label{thm:blowup}
		Let $u_2$ be the solution of \eqref{eq:blowup}. Then for sufficiently small $\varepsilon>0$, we have 
		\begin{equation*}
			|\nabla u_2(\bx)|\leq \frac{C}{\varepsilon+\lambda|x_1|^2}+C\big((k+1)\|u_{2}\|_{L^{2}(\mathbb{R}^2\setminus\overline{D})}+\|f_2\|_{L^{\infty}(\partial B_R)}\big),\quad \bx\in\mathbb{R}^2\setminus\overline{D},
		\end{equation*}
		and 
		\begin{equation*}
			|\nabla u_2(0,x_2)|\geq \frac{1}{C\varepsilon},
		\end{equation*}
		where $C>0$ is a constant independently of $\varepsilon$ and $\lambda$ is the curvature of $\partial D$ at $(0,\varepsilon/2)$ and $(0,-\varepsilon/2)$.
	\end{thm}

	In view of the gradient estimates for elliptic equations (see, for instance, \cite[Theorem 8.32]{GT1998}), for the solutions of \eqref{eq:blowup1} and \eqref{eq:blowup}, we have 
	\begin{equation*}
		\|\nabla u_1\|_{L^\infty(\mathbb R^2\setminus\overline B_R)}+\|\nabla u_2\|_{L^\infty(\mathbb R^2\setminus\overline B_R)}\leq C,
	\end{equation*}
	where $C>0$ is a constant independently of $\varepsilon$.
	As such, we shall consider the problem in $B_R\setminus\overline{D}=:\Omega$. 
	
	\subsection{Proof of Theorem \ref{thm:blowup1}}
	
	{
		\subsubsection{Main ingredients and outline of the proof of Theorem \ref{thm:blowup1}}
	}
	
	We decompose the solution of \eqref{eq:blowup1} as 
	$$u_1=v_{1,1}+v_{1,2},$$
	where $v_{1,1}$ and $v_{1,2}$, respectively, satisfy the following problem
	\begin{equation}\label{eq:blowup11}
		\begin{cases}
			\Delta v_{1,1}+k^2v_{1,1}=0\quad &\text{in} \quad \Omega,\\
			v_{1,1}=1+g(\bx)\quad &\text{on} \quad \partial D_1,\\
			v_{1,1}=1+h(\bx)\quad &\text{on} \quad \partial D_2,\\
			v_{1,1}=0\quad &\text{on} \quad \partial B_R, 
		\end{cases}
	\end{equation}
	and 
	\begin{equation}\label{eq:bounded112}
		\begin{cases}
			\Delta v_{1,2}+k^2v_{1,2}=0\quad &\text{in} \quad \Omega,\\
			v_{1,2}=0\quad &\text{on} \quad \partial D_1,\\
			v_{1,2}=0\quad &\text{on} \quad \partial D_2,\\
			v_{1,2}=f_1\quad &\text{on} \quad \partial B_R.
		\end{cases}
	\end{equation}
	
	\begin{prop}\label{prop-blowup11}
		Let $v_{1,1}$ be the solution of \eqref{eq:blowup11}. Then for sufficiently small $\varepsilon>0$, we have 
		\begin{align*}
			|\nabla v_{1,1}(\bx)|&\leq \frac{C|g(x_1,\varepsilon/2+\mathcal{H}_{1}(x_1))-h(x_1,-\varepsilon/2+\mathcal{H}_{2}(x_1))|}{\varepsilon+\lambda|x_1|^2}\\
			&\quad+C\big(\|g\|_{C^2(\partial D_1)}+\|h\|_{C^2(\partial D_2)}+k\|v_{1.1}\|_{L^2(\Omega)}+1\big),\quad \bx\in\Omega,
		\end{align*}
		where  $C>0$ is a constant independently of $\varepsilon$, and $\lambda$ is the curvature of $\partial D$ at $(0,\varepsilon/2)$ and $(0,-\varepsilon/2)$.
	\end{prop}
	
	\begin{prop}\label{prop-bounded12}
		Let $v_{1,2}$ be the solution of \eqref{eq:bounded112}. Then for sufficiently small $\varepsilon>0$, we have 
		\begin{equation*}
			|\nabla v_{1,2}(\bx)|\leq C\big((k+1)\|v_{1,2}\|_{L^{2}(\Omega)}+\|f_1\|_{L^{\infty}(\partial B_R)}\big),\quad \bx\in\Omega,
		\end{equation*}
		where $C>0$ is a constant independently of $\varepsilon$.
	\end{prop}
	
	{The proofs of Propositions \ref{prop-blowup11} and \ref{prop-bounded12} will be presented in Subsection \ref{subsec-prop}. With them in hand, we are ready to prove Theorem \ref{thm:blowup1}.
	}
	
	\begin{proof}[Proof of Theorem \ref{thm:blowup1}]
		By $u_1=v_{1,1}+v_{1,2}$ and Propositions \ref{prop-blowup11} and \ref{prop-bounded12}, we complete the proof of Theorem \ref{thm:blowup1}.
	\end{proof}
	
	{
		\subsubsection{Auxiliary lemmas}}
	
	Let us introduce an auxiliary function $p\in C^{2,\alpha}(\mathbb{R}^2)$, 
	\begin{equation}\label{def-p}
		p(\bx)=\frac{x_{2}+\frac{\varepsilon}{2}-\mathcal{H}_{2}(x_1)}{\delta(x_1)}\quad\hbox{in}\ \Omega_{2R_0},
	\end{equation}
	such that $p=1$ on $\partial{D}_{1}$, $p=0$ on $\partial{D}_{2}\cup\partial B_R$, and $\|p\|_{C^{2,\alpha}(\Omega\setminus\Omega_{R_0})}\leq\,C$, where $\Omega_{R_0}$ and $\delta(x_1)$ are defined in \eqref{narrowreg} and \eqref{delta_x'}, respectively. Now we use $p$ to construct an auxiliary function $\bar v_{1,1}$ such that $\bar v_{1,1}=1+g(\bx)$ on $\partial{D}_{1}$, $\bar v_{1,1}=1+h(\bx)$ on $\partial{D}_{2}$, $\bar v_{1,1}=0$ on $\partial B_R$,  
	\begin{equation*}
		\bar v_{1,1}=pg(x_1,\varepsilon/2+\mathcal{H}_{1}(x_1))+(1-p)h(x_1,-\varepsilon/2+\mathcal{H}_{2}(x_1))+1\quad\hbox{in}\ \Omega_{2R_0},
	\end{equation*}
	and $\|\bar v_{1,1}\|_{C^{2}(\Omega\setminus\Omega_{R_0})}\leq\,C(\|g\|_{C^2(\partial D_1)}+\|h\|_{C^2(\partial D_2)})$. By straightforward computations, we have, in $\Omega_{2R_0}$,
	\begin{align}\label{est-Dv11}
		|\partial_{x_1}\bar v_{1,1}(\bx)|&\leq \frac{C|x_1|}{\delta(x_1)}|g(x_1,\varepsilon/2+\mathcal{H}_{1}(x_1))-h(x_1,-\varepsilon/2+\mathcal{H}_{2}(x_1))|\nonumber\\
		&\quad+C\|\nabla g\|_{L^\infty(\partial D_1)}+C\|\nabla h\|_{L^\infty(\partial D_2)},
	\end{align}
	\begin{align}\label{est-Dv11-x2}
		|\partial_{x_2}\bar v_{1,1}(\bx)|&\leq \frac{C}{\delta(x_1)}|g(x_1,\varepsilon/2+\mathcal{H}_{1}(x_1))-h(x_1,-\varepsilon/2+\mathcal{H}_{2}(x_1))|\nonumber\\
		&\quad+C\|\nabla g\|_{L^\infty(\partial D_1)}+C\|\nabla h\|_{L^\infty(\partial D_2)},
	\end{align}
	and 
	\begin{align}\label{est-D2v11}
		|\nabla^2\bar v_{1,1}(\bx)|&\leq C\left(\frac{|x_1|}{\delta^2(x_1)}+\frac{1}{\delta(x_1)}\right)|g(x_1,\varepsilon/2+\mathcal{H}_{1}(x_1))-h(x_1,-\varepsilon/2+\mathcal{H}_{2}(x_1))|\nonumber\\
		&\quad+\frac{C}{\delta(x_1)}\left(\|\nabla g\|_{L^\infty(\partial D_1)}+\|\nabla h\|_{L^\infty(\partial D_2)}\right)\nonumber\\
		&\quad+C\left(\|\nabla^2g\|_{L^\infty(\partial D_1)}+\|\nabla^2h\|_{L^\infty(\partial D_2)}\right).
	\end{align}
	Next we shall investigate the gradient estimate of $v_{1,1}-\bar v_{1,1}$. For simplicity, we denote 
	\begin{equation*}
		v:=v_{1,1},\quad \bar v:=\bar v_{1,1},\quad\mbox{and}\quad w:=v-\bar v.
	\end{equation*}
	Then for small $k>0$, $w$ satisfies 
	\begin{equation}\label{eq:blowup-w11}
		\begin{cases}
			\Delta w+k^2w=-\Delta \bar v-k^2\bar v\quad &\text{in} \quad \Omega,\\
			w=0\quad &\text{on} \quad \partial D_1\cup\partial D_2,\\
			w=0\quad &\text{on} \quad \partial B_R.
		\end{cases}
	\end{equation}
	
	\begin{lem}\label{lem-global}
		Let $w$ be the solution of \eqref{eq:blowup-w11}. Then we have 
		\begin{align*}
			\|\nabla w\|_{L^2(\Omega)}&\leq C\Big(k\big(\|v\|_{L^2(\Omega)}+1+\|g\|_{L^\infty(\partial D_1)}+\|h\|_{L^\infty(\partial D_2)}\big)+\|g\|_{C^2(\partial D_1)}\\
			&\quad+\|h\|_{C^2(\partial D_2)}+1\Big),
		\end{align*}
		where  $C>0$ is a constant independently of $\varepsilon$.
	\end{lem}
	
	\begin{proof}
		Multiplying the equation in \eqref{eq:blowup-w11} by $w$ and integrating by parts, we have 
		\begin{align}\label{identity-Dw1}
			\int_{\Omega}|\nabla w|^2=\int_{\Omega}w(\Delta \bar v+k^2\bar v)+k^2\int_{\Omega}w^2.
		\end{align}
		It follows from the mean value theorem that, there exists $r_0\in(R_0/2,2R_0/3)$ such that 
		\begin{align*}
			\int_{\substack{|x_1|=r_0\\
					-\frac{\varepsilon}{2}+\mathcal{H}_{2}(x_1)<x_2<\frac{\varepsilon}{2}+\mathcal{H}_{1}(x_1)}}|w|\ dx_2&=\frac{6}{R_0}\int_{\substack{R_0/2<x_1<2R_0/3\\
					-\frac{\varepsilon}{2}+\mathcal{H}_{2}(x_1)<x_2<\frac{\varepsilon}{2}+\mathcal{H}_{1}(x_1)}}|w|\ d\bx\\
			&\leq CR_0^{5/2}\left(\int_{\Omega}|\nabla w|^2\right)^{1/2},
		\end{align*}
		where we used H\"{o}lder's inequality and the Poincar\'{e} inequality in the second inequality. Then combining H\"{o}lder's inequality and \eqref{est-Dv11}, we obtain
		\begin{align}\label{est-wbar-v}
			\int_{\Omega_{r_0}}w\Delta \bar v&=\int_{\Omega_{r_0}}w\partial_{x_1}^2\bar v\nonumber\\
			&=-\int_{\Omega_{r_0}}\partial_{x_1}w\partial_{x_1}\bar v+\int_{\substack{|x_1|=r_0\nonumber\\
					-\frac{\varepsilon}{2}+\mathcal{H}_{2}(x_1)<x_2<\frac{\varepsilon}{2}+\mathcal{H}_{1}(x_1)}}w\partial_{x_1}\bar v\frac{x_1}{r_0}\ ds\nonumber\\
			&\leq\left(\int_{\Omega_{r_0}}|\partial_{x_1}w|^2\right)^{1/2}\left(\int_{\Omega_{r_0}}|\partial_{x_1}\bar v|^2\right)^{1/2}+C\left(\int_{\Omega}|\nabla w|^2\right)^{1/2}\nonumber\\
			&\leq C(\|g\|_{C^1(\partial D_1)}+\|h\|_{C^1(\partial D_2)}+1)\left(\int_{\Omega}|\nabla w|^2\right)^{1/2}.
		\end{align}
		Substituting \eqref{est-wbar-v} into \eqref{identity-Dw1} and using $w=v-\bar v$, we have 
		\begin{align*}
			\int_{\Omega}|\nabla w|^2&\leq\int_{\Omega}w(\Delta \bar v+k^2\bar v)+k^2\int_{\Omega}w^2\\
			&\leq\int_{\Omega_{r_0}}w\Delta \bar v+\int_{\Omega\setminus\Omega_{r_0}}w\Delta \bar v+Ck^2\int_{\Omega}v^2+Ck^2(1+\|g\|_{L^\infty(\partial D_1)}^2+\|h\|_{L^\infty(\partial D_2)}^2)\\
			&\leq C(\|g\|_{C^1(\partial D_1)}+\|h\|_{C^1(\partial D_2)}+1)\left(\int_{\Omega}|\nabla w|^2\right)^{1/2}\\
			&\quad+C(\|g\|_{C^2(\partial D_1)}+\|h\|_{C^2(\partial D_2)}+1)\int_{\Omega\setminus\Omega_{r_0}}|w|+Ck^2\int_{\Omega}v^2\\
			&\quad+Ck^2(1+\|g\|_{L^\infty(\partial D_1)}^2+\|h\|_{L^\infty(\partial D_2)}^2)\\
			&\leq C(\|g\|_{C^2(\partial D_1)}+\|h\|_{C^2(\partial D_2)}+1)\left(\int_{\Omega}|\nabla w|^2\right)^{1/2}+Ck^2\int_{\Omega}v^2\\
			&\quad+Ck^2(1+\|g\|_{L^\infty(\partial D_1)}^2+\|h\|_{L^\infty(\partial D_2)}^2),
		\end{align*}
		where we used the Poincar\'{e} inequality in the last inequality. This yields the desired result.
	\end{proof}
	
	For any $\bz=(z_1,z_2)\in\Omega_{R_0}$, we set
	\begin{equation*}
		\Omega_{s}(z_1):=\left\{\bx\in\mathbb R^2: -\frac{\varepsilon}{2}+\mathcal{H}_{2}(x_1)<x_2<\frac{\varepsilon}{2}+\mathcal{H}_{1}(x_1),~|x_1-z_1|<s\right\}.
	\end{equation*}
	
	\begin{lem}\label{lem-localbdd}
		Let $w$ be the solution of \eqref{eq:blowup-w11}. Then for $\bz=(z_1,z_2)\in\Omega_{R_0}$ and sufficiently small $\varepsilon>0$, we have 
		\begin{align*}
			\|\nabla w\|_{L^2(\Omega_{\delta(z_1)}(z_1))}&\leq C\sqrt{\delta(z_1)}\Big[|g(z_1,\varepsilon/2+\mathcal{H}_{1}(z_1))-h(z_1,-\varepsilon/2+\mathcal{H}_{2}(z_1))|\nonumber\\
			&\quad+\sqrt{\delta(z_1)}\big(\|g\|_{C^2(\partial D_1)}+\|h\|_{C^2(\partial D_2)}+k\|v\|_{L^2(\Omega)}+1\big)\Big],
		\end{align*}
		where $C>0$ is a constant independently of $\varepsilon$.
	\end{lem}
	
	\begin{proof}
		For $0<t<s\leq R_0$, let $\eta$ be a smooth function satisfying $\eta(x_1)=1$ if $|x_1-z_1|<t$, $\eta(x_1)=0$ if $|x_1-z_1|>s$, $0\leqslant\eta(x_1)\leqslant1$ if $t\leqslant|x_1-z_1|<s$,
		and
		$|\eta'(x_1)\leq \frac{2}{s-t}$. Multiplying the equation in \eqref{eq:blowup-w11} by $\eta^{2}w$, integrating by parts, and using Young's inequality and $|\bar v|\leq C(1+\|g\|_{L^\infty(\partial D_1)}+\|h\|_{L^\infty(\partial D_2)})$, we have
		\begin{align}\label{local-Dw1}
			&\int_{\Omega_{s}(z_1)}\eta^2|\nabla w|^2=\int_{\Omega_{s}(z_1)}\eta^2w(k^2w+\Delta\bar v+k^2\bar v)-2\int_{\Omega_{s}(z_1)}\eta\nabla w\cdot w\nabla\eta\nonumber\\
			&\leq Ck^2\int_{\Omega_{s}(z_1)}|w|^2+\frac{C}{(s-t)^2}\int_{\Omega_{s}(z_1)}|w|^2+C(s-t)^2\int_{\Omega_{s}(z_1)}|\Delta\bar v|^2\nonumber\\
			&\quad+\frac{Ck^2}{(s-t)^2}\int_{\Omega_{s}(z_1)}|w|^2+Ck^2(s-t)^2|\Omega_{s}(z_1)|(1+\|g\|_{L^\infty(\partial D_1)}^2+\|h\|_{L^\infty(\partial D_2)}^2)\nonumber\\
			&\quad+\frac{1}{2}\int_{\Omega_{s}(z_1)}\eta^2|\nabla w|^2.
		\end{align}
		Since $w=0$ on $\partial D_1\cup\partial D_2$, by the Poincar\'e inequality, we derive
		\begin{align*}
			\int_{\Omega_{s}(z_1)}|w|^2\leq
			C\delta^2(z_1)\int_{\Omega_{s}(z_1)}|\nabla w|^2.
		\end{align*}
		Then \eqref{local-Dw1} becomes
		\begin{align}\label{local-Dw1-int}
			&\int_{\Omega_{s}(z_1)}\eta^2|\nabla w|^2\nonumber\\
			&\leq \left(C_0\delta^2(z_1)k^2+\frac{c_0\delta^2(z_1)(k^2+1)}{(s-t)^2}\right)\int_{\Omega_{s}(z_1)}|\nabla w|^2+C(s-t)^2\int_{\Omega_{s}(z_1)}|\Delta\bar v|^2\nonumber\\
			&\quad+Ck^2(s-t)^2|\Omega_{s}(z_1)|(1+\|g\|_{L^\infty(\partial D_1)}^2+\|h\|_{L^\infty(\partial D_2)}^2),
		\end{align}
		where $C_0, c_{0}>0$ are constants independently of $\varepsilon$. By shrinking $\varepsilon$ and $R_0$,  we have  
		\begin{equation}\label{small-term}
			\delta(z_1)\leq 1\quad\mbox{and}\quad C_0\delta^2(z_1)k^2\leq \frac{1}{4}.
		\end{equation}
		Substituting \eqref{small-term} into \eqref{local-Dw1-int}, we obtain
		\begin{align}\label{ite-Dw1}
			\int_{\Omega_{s}(z_1)}\eta^2|\nabla w|^2&\leq \left(\frac{1}{4}+\frac{c_0\delta^2(z_1)(k^2+1)}{(s-t)^2}\right)\int_{\Omega_{s}(z_1)}|\nabla w|^2+C(s-t)^2\int_{\Omega_{s}(z_1)}|\Delta\bar v|^2\nonumber\\
			&\quad+Ck^2(s-t)^2s(1+\|g\|_{L^\infty(\partial D_1)}^2+\|h\|_{L^\infty(\partial D_2)}^2).
		\end{align}
		By using \eqref{est-D2v11} and the mean value theorem, we have 
		\begin{align*}
			\int_{\Omega_{s}(z_1)}|\Delta\bar v|^2
			&\leq C|g(z_1,\varepsilon/2+\mathcal{H}_{1}(z_1))-h(z_1,-\varepsilon/2+\mathcal{H}_{2}(z_1))|^2\int_{\Omega_{s}(z_1)}\left(\frac{|x_1|}{\delta^2(x_1)}+\frac{1}{\delta(x_1)}\right)^2\ d\bx\\
			&\quad+C\left(\|\nabla g\|_{L^\infty(\partial D_1)}^2+\|\nabla h\|_{L^\infty(\partial D_2)}^2\right)\int_{\Omega_{s}(z_1)}\left(\frac{|x_1|}{\delta^2(x_1)}+\frac{1}{\delta(x_1)}\right)^2|x_1-z_1|^2\ d\bx\\
			&\quad+C\left(\|\nabla g\|_{L^\infty(\partial D_1)}^2+\|\nabla h\|_{L^\infty(\partial D_2)}^2\right)\int_{\Omega_{s}(z_1)}\frac{1}{\delta^2(x_1)}\ dx\\
			&\quad+Cs\left(\|\nabla^2 g\|_{L^\infty(\partial D_1)}^2+\|\nabla^2 h\|_{L^\infty(\partial D_2)}^2\right)\\                                 
			&\leq Cs\Bigg[\frac{|g(z_1,\varepsilon/2+\mathcal{H}_{1}(z_1))-h(z_1,-\varepsilon/2+\mathcal{H}_{2}(z_1))|^2}{\delta^2(z_1)}\\
			&\qquad\qquad+\Big(\|\nabla g\|_{L^\infty(\partial D_1)}^2+\|\nabla h\|_{L^\infty(\partial D_2)}^2\Big)\Big(\frac{s^2}{\delta^2(z_1)}+\frac{1}{\delta(z_1)}\Big)\\
			&\qquad\qquad+\Big(\|\nabla^2 g\|_{L^\infty(\partial D_1)}^2+\|\nabla^2 h\|_{L^\infty(\partial D_2)}^2\Big)\Bigg].
		\end{align*}
		This together with \eqref{ite-Dw1} gives
		\begin{align}\label{iteration}
			\mathcal{F}(t)&\leq \left(\frac{1}{4}+\frac{c_0\delta^2(z_1)(k^2+1)}{(s-t)^2}\right)\mathcal{F}(s)\nonumber\\
			&\quad+C(s-t)^2s\Bigg[\frac{|g(z_1,\varepsilon/2+\mathcal{H}_{1}(z_1))-h(z_1,-\varepsilon/2+\mathcal{H}_{2}(z_1))|^2}{\delta^2(z_1)}\nonumber\\
			&\quad+\Big(\|\nabla g\|_{L^\infty(\partial D_1)}^2+\|\nabla h\|_{L^\infty(\partial D_2)}^2\Big)\Big(\frac{s^2}{\delta^2(z_1)}+\frac{1}{\delta(z_1)}\Big)\nonumber\\
			&\quad+\Big(\|g\|_{C^2(\partial D_1)}^2+\|h\|_{C^2(\partial D_2)}^2+k^2\Big)\Bigg],
		\end{align}
		where
		$$\mathcal{F}(t):=\int_{\Omega_{t}(z_1)}|\nabla w|^{2}.$$  Let $m_{0}=\lceil\frac{1}{4\sqrt{c_{0}(k^2+1)\delta(z_1)}}\rceil+1$ and $t_{i}=\delta(z_1)+2\sqrt{c_{0}(k^2+1)}i\delta(z_1), i=0,1,2,\dots,m_{0}$. Then we take $s=t_{i+1}$ and $t=t_{i}$ in \eqref{iteration} to get 
		\begin{align*}
			\mathcal{F}(t_{i})&\leq \frac{1}{2}\mathcal{F}(t_{i+1})+C(i+1)(k^2+1)\delta(z_1)\Big[|g(z_1,\varepsilon/2+\mathcal{H}_{1}(z_1))-h(z_1,-\varepsilon/2+\mathcal{H}_{2}(z_1))|^2\nonumber\\
			&\quad+\delta(z_1)\big(\|g\|_{C^2(\partial D_1)}^2+\|h\|_{C^2(\partial D_2)}^2+k^2\big)\Big].
		\end{align*}
		After $m_{0}$ iterations, and using Lemma \ref{lem-global}, for sufficiently small $\varepsilon$ and $|z_1|$, we obtain
		\begin{align*}
			\mathcal{F}(t_0)&\leq \left(\frac{1}{2}\right)^{m_0}\mathcal{F}(t_{m_0})
			+C(k^2+1)\delta(z_1)\Big[|g(z_1,\varepsilon/2+\mathcal{H}_{1}(z_1))-h(z_1,-\varepsilon/2+\mathcal{H}_{2}(z_1))|^2\nonumber\\
			&\quad+\delta(z_1)\big(\|g\|_{C^2(\partial D_1)}^2+\|h\|_{C^2(\partial D_2)}^2+k^2\big)\Big]\sum\limits_{l=0}^{m_0-1}\left(\frac{1}{2}\right)^{l}(l+1)\\
			&\leq C\delta(z_1)\Big[|g(z_1,\varepsilon/2+\mathcal{H}_{1}(z_1))-h(z_1,-\varepsilon/2+\mathcal{H}_{2}(z_1))|^2\nonumber\\
			&\quad+\delta(z_1)\big(\|g\|_{C^2(\partial D_1)}^2+\|h\|_{C^2(\partial D_2)}^2+k^2\|v\|_{L^2(\Omega)}^2+1\big)\Big].
		\end{align*}
		This finishes the proof.
	\end{proof}
	
	{\subsubsection{Proofs of Propositions \ref{prop-blowup11} and \ref{prop-bounded12}}\label{subsec-prop}}
	
	With Lemma \ref{lem-localbdd} in hand, we are ready to finish the proof of Propositions \ref{prop-blowup11} and \ref{prop-bounded12}.
	
	\begin{proof}[Proof of Proposition \ref{prop-blowup11}]
		For simplicity, we denote $\delta:=\delta(z_1)$. Taking the following change of variables
		\begin{equation*}
			\left\{
			\begin{aligned}
				&x_1-z_1=\delta y_1,\\
				&x_2=\delta y_{2}.
			\end{aligned}
			\right.
		\end{equation*}
		Then we transform $\Omega_{\delta}(z_1)$ into a nearly unit size domain 
		\begin{align*}
			Q_{1}:=\Big\{\by\in\mathbb{R}^{2}: -\frac{\varepsilon}{2\delta}-\frac{1}{\delta}\mathcal{H}_{2}(\delta\,y_1+z_1)<y_{2}
			<\frac{\varepsilon}{2\delta}+\frac{1}{\delta}\mathcal{H}_{1}(\delta\,y_1+z_1),~|y_1|<1\Big\}.
		\end{align*}
		Denote its top and bottom boundaries by $\Gamma_{1}^{+}$ and $\Gamma_{1}^{-}$, respectively. Let us define
		\begin{equation*}
			W(y_1, y_{2}):=w(\delta\,y_1+z_1,\delta\,y_{2}),\quad \bar V(y_1, y_{2}):=\bar v(\delta\,y'+z',\delta\,y_{2}),\quad \by\in{Q}_{1}.
		\end{equation*}
		Then it follows from \eqref{eq:blowup-w11} that $W$ satisfies 
		\begin{equation*}
			\begin{cases}
				\Delta W+k^2\delta^2W=-\Delta \bar V-k^2\delta^2\bar V\quad &\text{in} \quad \Omega,\\
				W=0\quad &\text{on} \quad \Gamma_{1}^{+}\cup\Gamma_{1}^{+}.
			\end{cases}
		\end{equation*}
		By employing the standard interior and boundary $L^p$-estimate (see, for instance, \cite[Theorem 9.37]{GT1998}), the Poincar\'{e} inequality, and the Sobolev embedding theorem, we obtain, for $p=3$,
		\begin{align*}
			\|\nabla W\|_{L^\infty(Q_{1/2})}&\leq C\|W\|_{W^{2,p}(Q_{1/2})}\\
			&\leq C\left(\|W\|_{L^{p}(Q_{1})}+k^2\delta^2\|\bar V\|_{L^{\infty}(Q_{1})}+\|\Delta \bar V\|_{L^{\infty}(Q_{1})}\right)\\
			&\leq C\left(\|\nabla W\|_{L^{2}(Q_{1})}+k^2\delta^2\|\bar V\|_{L^{\infty}(Q_{1})}+\|\Delta \bar V\|_{L^{\infty}(Q_{1})}\right).
		\end{align*}
		Rescaling back, we derive 
		\begin{align*}
			\|\nabla w\|_{L^\infty(\Omega_{\delta}(z_1))}
			&\leq C\delta^{-1}\left(\|\nabla w\|_{L^{2}(\Omega_{\delta}(z_1))}+k^2\delta^2\|\bar v\|_{L^{\infty}(\Omega_{\delta}(z_1))}+\delta^2\|\Delta \bar v\|_{L^{\infty}(\Omega_{\delta}(z_1))}\right).
		\end{align*}
		This in combination with Lemma \ref{lem-localbdd} and \eqref{est-D2v11} yields 
		\begin{align}\label{est-gradientw}
			\|\nabla w\|_{L^\infty(\Omega_{\delta(z_1)/2}(z_1))}&\leq \frac{C|g(z_1,\varepsilon/2+\mathcal{H}_{1}(z_1))-h(z_1,-\varepsilon/2+\mathcal{H}_{2}(z_1))|}{\sqrt{\delta(z_1)}}\nonumber\\
			&\quad+C\big(\|g\|_{C^2(\partial D_1)}+\|h\|_{C^2(\partial D_2)}+k\|v\|_{L^2(\Omega)}+1\big).
		\end{align}
		Thus, by using $v_{1,1}=w+\bar v_{1,1}$, \eqref{est-Dv11}, and \eqref{est-Dv11-x2}, we obtain, for $\bx\in\Omega_{2R_0}$,
		\begin{align*}
			|\nabla v_{1,1}(\bx)|&\leq \frac{C|g(x_1,\varepsilon/2+\mathcal{H}_{1}(x_1))-h(x_1,-\varepsilon/2+\mathcal{H}_{2}(x_1))|}{\varepsilon+\lambda|x_1|^2}\\
			&\quad+C\big(\|g\|_{C^2(\partial D_1)}+\|h\|_{C^2(\partial D_2)}+k\|v_{1.1}\|_{L^2(\Omega)}+1\big).
		\end{align*}
		Note that $\Omega\setminus\Omega_{2R_0}\subset\Omega\setminus\Omega_{R_0}$. By using the
		Sobolev embedding theorem and \cite[Theorem 9.37]{GT1998}, we have, for $p=3$,
		\begin{align}\label{est-out-Dv11}
			\|\nabla v_{1,1}\|_{L^\infty(\Omega\setminus\Omega_{2R_0})}&\leq C\|v_{1,1}\|_{W^{2,p}(\Omega\setminus\Omega_{2R_0})}\nonumber\\
			&\leq C\big(\|v_{1,1}\|_{L^{2}(\Omega\setminus\Omega_{R_0})}+\|g\|_{L^{\infty}(\partial D_1)}+\|h\|_{L^{\infty}(\partial D_2)}\big).
		\end{align}
		Therefore, Proposition \ref{prop-blowup11} is proved.
	\end{proof}
	
	\begin{proof}[Proof of Proposition \ref{prop-bounded12}]
		As in the proof of \eqref{est-out-Dv11}, for $p>2$, we have, 
		\begin{align*}
			\|\nabla v_{1,2}\|_{L^\infty(\Omega\setminus\Omega_{2R_0})}&\leq C\|v_{1,2}\|_{W^{2,p}(\Omega\setminus\Omega_{2R_0})}\nonumber\\
			&\leq C\big(\|v_{1,2}\|_{L^{2}(\Omega\setminus\Omega_{R_0})}+\|f_1\|_{L^{\infty}(\partial B_R)}\big).
		\end{align*}
		Thus, it suffices to consider the problem 
		\begin{equation*}
			\begin{cases}
				\Delta v_{1,2}+k^2v_{1,2}=0\quad &\text{in} \quad \Omega_{2R_0},\\
				v_{1,2}=0\quad &\text{on} \quad \Gamma_+^{2R_0}\cup\Gamma_-^{2R_0},
			\end{cases}
		\end{equation*}
		and estimate the gradient of $v_{1,2}$ in the narrow region $\Omega_{2R_0}$. Here, 
		\begin{equation}\label{def-gamma+R}
			\Gamma_+^{2R_0}:=\left\{(x_1,x_2)\in\mathbb R^2:  ~x_2=\frac{\varepsilon}{2}+\mathcal{H}_{1}(x_1),~|x_1|\leq 2R_0\right\}
		\end{equation}
		and 
		\begin{equation}\label{def-gamma-R}
			\Gamma_-^{2R_0}:=\left\{(x_1,x_2)\in\mathbb R^2:  ~x_2=-\frac{\varepsilon}{2}+\mathcal{H}_{2}(x_1),~|x_1|\leq 2R_0\right\}.
		\end{equation}
		To this end, by multiplying the equation in \eqref{eq:bounded112} by $v_{1,2}$ and integrating by parts, we have 
		\begin{align}\label{Dv12-omega}
			\int_{\Omega_{2R_0}}|\nabla v_{1,2}|^2&=k^2\int_{\Omega_{2R_0}}|v_{1,2}|^2+\int_{\substack{|x_1|=2R_0\\
					-\frac{\varepsilon}{2}+\mathcal{H}_{2}(x_1)<x_2<\frac{\varepsilon}{2}+\mathcal{H}_{1}(x_1)}}v_{1,2}\nabla v_{1,2}\frac{\bx}{r}\nonumber\\
			&\leq k^2\int_{\Omega_{2R_0}}|v_{1,2}|^2+\int_{\substack{|x_1|=2R_0\\
					-\frac{\varepsilon}{2}+\mathcal{H}_{2}(x_1)<x_2<\frac{\varepsilon}{2}+\mathcal{H}_{1}(x_1)}}(|v_{1,2}|^2+|\nabla v_{1,2}|^2).
		\end{align}
		Note that $\Omega_{3R_0}\setminus\Omega_{3R_0/2}\subset\Omega\setminus\Omega_{R_0}$. Similar to the proof of \eqref{est-out-Dv11}, we obtain 
		\begin{align*}
			\|\nabla v_{1,2}\|_{L^\infty(\Omega_{3R_0}\setminus\Omega_{3R_0/2})}&\leq C\|v_{1,2}\|_{W^{2,p}(\Omega_{3R_0}\setminus\Omega_{3R_0/2})}\nonumber\\
			&\leq C\big(\|v_{1,2}\|_{L^{2}(\Omega\setminus\Omega_{R_0})}+\|f_1\|_{L^{\infty}(\partial B_R)}\big).
		\end{align*}
		Recalling that $v_{1,2}=0$ on $\partial D_1$, we have, for any $\bx\in\Omega_{3R_0}\setminus\Omega_{3R_0/2}$,
		\begin{align*}
			|v_{1,2}(x_1,x_2)|&=|v_{1,2}(x_1,x_2)-v_{1,2}(x_1,\frac{\varepsilon}{2}+\mathcal{H}_{1}(x_1))|\\
			&\leq C(\varepsilon+\lambda|x_1|^2)\|\nabla v_{1,2}\|_{L^\infty(\Omega_{3R_0}\setminus\Omega_{3R_0/2})}\\
			&\leq C\big(\|v_{1,2}\|_{L^{2}(\Omega\setminus\Omega_{R_0})}+\|f_1\|_{L^{\infty}(\partial B_R)}\big).
		\end{align*}
		Hence, we derive 
		\begin{align*}
			\int_{\substack{|x_1|=2R_0\nonumber\\
					-\frac{\varepsilon}{2}+\mathcal{H}_{2}(x_1)<x_2<\frac{\varepsilon}{2}+\mathcal{H}_{1}(x_1)}}(|v_{1,2}|^2+|\nabla v_{1,2}|^2)\leq C\big(\|v_{1,2}\|_{L^{2}(\Omega\setminus\Omega_{R_0})}^2+\|f_1\|_{L^{\infty}(\partial B_R)}^2\big).
		\end{align*}
		Coming back to \eqref{Dv12-omega}, we have
		\begin{equation*}
			\|\nabla v_{1,2}\|_{L^2(\Omega_{2R_0})}\leq C\big((k+1)\|v_{1,2}\|_{L^{2}(\Omega)}+\|f_1\|_{L^{\infty}(\partial B_R)}\big).
		\end{equation*}
		By adapting a similar argument that led to Lemma \ref{lem-localbdd} with $\bar v=0$, we have 
		\begin{equation*}
			\|\nabla v_{1,2}\|_{L^2(\Omega_{\delta(z_1)}(z_1))}\leq C\big((k+1)\|v_{1,2}\|_{L^{2}(\Omega)}+\|f_1\|_{L^{\infty}(\partial B_R)}\big)\delta(z_1),\quad \bz=(z_1,z_2)\in\Omega_{2R_0}.
		\end{equation*}
		Therefore, by following the proof of Proposition \ref{prop-blowup11}, we have 
		\begin{equation*}
			\|\nabla v_{1,2}\|_{L^\infty(\Omega_{\delta(z_1)/2}(z_1))}\leq C\big((k+1)\|v_{1,2}\|_{L^{2}(\Omega)}+\|f_1\|_{L^{\infty}(\partial B_R)}\big),\quad \bz=(z_1,z_2)\in\Omega_{2R_0}.
		\end{equation*}
		Proposition \ref{prop-bounded12} is proved.
	\end{proof}

	\subsection{Proof of Theorem \ref{thm:blowup}}
	As in the proof of Theorem \ref{thm:blowup1}, we decompose the solution of \eqref{eq:blowup} as 
	$$u_2=v_{2,1}+v_{2,2},$$
	where $v_{2,1}$ and $v_{2,2}$ satisfy the following problem, respectively,
	\begin{equation}\label{eq:blowup21}
		\begin{cases}
			\Delta v_{2,1}+k^2v_{2,1}=0\quad &\text{in} \quad \Omega,\\
			v_{2,1}=1+o(1)\quad &\text{on} \quad \partial D_1,\\
			v_{2,1}=-1+o(1)\quad &\text{on} \quad \partial D_2,\\
			v_{2,1}=0\quad &\text{on} \quad \partial B_R,
		\end{cases}
	\end{equation}
	and 
	\begin{equation}\label{eq:bounded2}
		\begin{cases}
			\Delta v_{2,2}+k^2v_{2,2}=0\quad &\text{in} \quad \Omega,\\
			v_{2,2}=0\quad &\text{on} \quad \partial D_1,\\
			v_{2,2}=0\quad &\text{on} \quad \partial D_2,\\
			v_{2,2}=f_2\quad &\text{on} \quad \partial B_R.
		\end{cases}
	\end{equation}
	
	We adapt the argument in the proof of Proposition \ref{prop-blowup11} to prove the gradient estimates of the solution of \eqref{eq:blowup21}. To this end, we construct an auxiliary function $\bar v_{2,1}\in C^{2,\alpha}(\mathbb{R}^2)$, such that $\bar v_{2,1}=1+o(1)$ on $\partial{D}_{1}$, $\bar v_{2,1}=-1+o(1)$ on $\partial{D}_{2}$, $\bar v_{2,1}=0$ on $\partial B_R$,
	\begin{equation*}
		\bar v_{2,1}(\bx)=p(1+o(1))+(1-p)(-1+o(1))\quad\hbox{in}\ \Omega_{2R_0},
	\end{equation*}
	and $\|\bar v_{2,1}\|_{C^{2,\alpha}(D^e\setminus\Omega_{R_0})}\leq\,C$, where $p$ is defined in \eqref{def-p}. {Then by using a direct calculation, we have 
		\begin{equation}\label{est-barv21}
			|\partial_{x_1}\bar v_{2,1}|\leq \frac{C|x_1|}{\delta(x_1)}\quad\mbox{and}\quad \partial_{x_2}\bar v_{2,1}=\frac{2+o(1)}{\delta(x_1)}.
		\end{equation}
	} Denote 
	\begin{equation}\label{def-w21}
		w_{2,1}:=v_{2,1}-\bar v_{2,1}.
	\end{equation}
	Then $w_{2,1}$ satisfies \eqref{eq:blowup-w11}. Therefore, by mimicking the arguments of the proof of {\eqref{est-gradientw}, we derive 
		\begin{equation*}
			\|\nabla w_{2,1}\|_{L^\infty(\Omega_{\delta(x_1)/2}(x_1))}\leq \frac{C}{\sqrt{\delta(x_1)}}+C\big(k\|v_{2,1}\|_{L^2(\Omega)}+1\big).
		\end{equation*}
		This in combination with \eqref{est-barv21} and \eqref{def-w21} yields } the result in Proposition \ref{prop-blowup21} below. The details are omitted here.
	
	\begin{prop}\label{prop-blowup21}
		Let $v_{2,1}$ be the solution of \eqref{eq:blowup21}, then for sufficiently small $\varepsilon>0$, we have 
		\begin{equation*}
			|\nabla v_{2,1}(\bx)|\leq \frac{C}{\varepsilon+\lambda|x_1|^2}{+C\big(k\|v_{2.1}\|_{L^2(\Omega)}+1\big)},\quad \bx\in\Omega,
		\end{equation*}
		and 
		\begin{equation*}
			|\nabla v_{2,1}(0,x_2)|\geq \frac{1}{C\varepsilon},\quad{x_2\in(-\varepsilon/2,\varepsilon/2),}
		\end{equation*}
		where $C>0$ is a constant independently of $\varepsilon$ and $\lambda$ is the curvature of $\partial D$ at $(0,\varepsilon/2)$ and $(0,-\varepsilon/2)$.
	\end{prop}
	
	Note that the problem \eqref{eq:bounded2} is similar to \eqref{eq:bounded112} with $f_1$ replaced by $f_2$. Thus, the proof of Proposition \ref{prop-bounded2} below is the same as that of Proposition \ref{prop-bounded12}.
	
	\begin{prop}\label{prop-bounded2}
		Let $v_{2,2}$ be the solution of \eqref{eq:bounded2}, then for sufficiently small $\varepsilon>0$, we have 
		\begin{equation*}
			|\nabla v_{2,2}(\bx)|\leq C\big((k+1)\|v_{2,2}\|_{L^{2}(\Omega)}+\|f_2\|_{L^{\infty}(\partial B_R)}\big),\quad \bx\in\Omega,
		\end{equation*}
		where $C>0$ is a constant independently of $\varepsilon$.
	\end{prop}
	
	\begin{proof}[Proof of Theorem \ref{thm:blowup}]
		By $u_2=v_{2,1}+v_{2,2}$ and Propositions \ref{prop-blowup21} and \ref{prop-bounded2}, Theorem \ref{thm:blowup} follows.
	\end{proof} 
	
	\appendix
	
	\section{Appendix}\label{Appendix}
	In this section, we analyze the estimate of $\alpha_{12}-\alpha_{11}$ defined in \eqref{eq:deaij} when the two resonators get closer to each other. It follows from \eqref{form-alphaij} and $\alpha_{ij}=\alpha_{ji}$ given in Lemma \ref{lem:ralij}  that 
	\begin{align*}
		\alpha_{12}-\alpha_{11}&=\int_{\partial D_1}\partial_{\bnu}(w_2-w_1)|_{+}\ ds\\
		&=\int_{\partial D_1\cap\mathcal{C}_{R_0}}\partial_{\bnu}(w_2-w_1)|_{+}\ ds+\int_{\partial D_1\setminus\mathcal{C}_{R_0}}\partial_{\bnu}(w_2-w_1)|_{+}\ ds,
	\end{align*}
	where $R_0$ is given in \eqref{h1h2}, and $\mathcal{C}_{R_0}$ is defined by
	\begin{equation*}
		\mathcal{C}_{R_0}:=\left\{(x_1,x_2)\in \mathbb{R}^{2}: -\frac{\varepsilon}{2}+2\min_{|x_1|=R_0}\mathcal{H}_{2}(x_1)<x_2<\frac{\varepsilon}{2}+2\max_{|x_1|=R}\mathcal{H}_{1}(x_1),~|x_1|<r\right\}.
	\end{equation*}
	Note that the components of the normal vector $\bnu$ on the portion $\partial D_1\cap\mathcal{C}_{R_0}$ are 
	\begin{equation*}
		\nu_1=\frac{\partial_{x_1}\mathcal{H}_{1}(x_1)}{\sqrt{1+|\partial_{x_1}\mathcal{H}_{1}(x_1)|^2}},\quad\nu_2=-\frac{1}{\sqrt{1+|\partial_{x_1}\mathcal{H}_{1}(x_1)|^2}}.
	\end{equation*}
	Then combining with the gradient estimates for elliptic equations, we have 
	\begin{equation}\label{alpha12-11}
		\alpha_{12}-\alpha_{11}=\int_{|x_1|\leq R_0}\big(\partial_{x_1}(w_2-w_1)\partial_{x_1}\mathcal{H}_{1}(x_1)-\partial_{x_2}(w_2-w_1)\big)\ dx_1+\Ocal(1).
	\end{equation}
	Thus, we need to estimate the gradient of $w_2-w_1$ in $\Omega_{R_0}$, where $\Omega_{R_0}$ is defined in \eqref{narrowreg}. To this end, we construct an auxiliary function $\bar w\in C^{2,\alpha}(\mathbb{R}^2)$, such that $\bar w=-1$ on $\partial{D}_{1}$, $\bar w=1$ on $\partial{D}_{2}$,
	\begin{equation}\label{def-barw}
		\bar w(\bx)=-2p(\bx)+1\quad\hbox{in}\ \Omega_{2R_0},
	\end{equation}
	and $\|\bar w\|_{C^{2,\alpha}(\Omega\setminus\Omega_{R_0})}\leq\,C$, where $p$  is defined in \eqref{def-p}. Then $w_2-w_1-\bar w$ satisfies 
	\begin{align*}
		\begin{cases}
			\Delta(w_2-w_1-\bar w)=-\Delta\bar w\quad &\text{in} \quad \Omega_{2R_0},\\
			w_2-w_1-\bar w=0\quad &\text{on} \quad \Gamma_+^{2R_0}\cup\Gamma_-^{2R_0},
		\end{cases}
	\end{align*}
	where $\Gamma_+^{2R_0}$ and $\Gamma_-^{2R_0}$ are defined in \eqref{def-gamma+R} and \eqref{def-gamma-R}, respectively. Similar to the proof of \eqref{est-gradientw}, we obtain 
	\begin{equation*}
		\|\nabla(w_2-w_1-\bar w)\|_{L^\infty(\Omega_{2R_0})}\leq \frac{C}{\sqrt{\delta(x_1)}},
	\end{equation*}
	where $\delta(x_1)$ is defined in \eqref{delta_x'}. This together with \eqref{def-barw} and \eqref{def-p}  yields 
	\begin{align}\label{est-D1w2-w1}
		|\partial_{x_1}(w_2-w_1)|\leq\frac{C}{\sqrt{\delta(x_1)}}
	\end{align}
	and 
	\begin{align}\label{est-D2w2-w1}
		\partial_{x_2}(w_2-w_1)=-\frac{2}{\delta(x_1)}+\frac{1}{\sqrt{\delta(x_1)}}\Ocal(1).
	\end{align}
	It follows from \eqref{delta_x'} that
	\begin{align}\label{est-delta-alpha}
		\left|\frac{1}{\delta(x_1)}-\frac{1}{\varepsilon+\lambda|x_1|^2}\right|\leq\frac{C|x_1|^\alpha}{\varepsilon+\lambda|x_1|^2}.
	\end{align}
	Note that
	\begin{align}\label{est-alpha}
		\int_{|x_1|\leq R_0}\frac{|x_1|^\alpha}{\varepsilon+\lambda|x_1|^2}dx_1
		&=2\int_0^{R_0}\frac{r^\alpha}{\varepsilon+\lambda r^2}\ dr\nonumber\\
		&=2\varepsilon^{\frac{\alpha-1}{2}}\lambda^{-\frac{\alpha+1}{2}}\int_{0}^{R_0\sqrt{\frac{\lambda}{\varepsilon}}}\frac{r^\alpha}{1+r^2}\ dr=\varepsilon^{\frac{\alpha-1}{2}}\Ocal(1).
	\end{align}
	Substituting \eqref{est-D1w2-w1} and \eqref{est-D2w2-w1} into \eqref{alpha12-11}, and using \eqref{est-delta-alpha} and \eqref{est-alpha}, we derive 
	\begin{align*}
		\alpha_{12}-\alpha_{11}&=-\int_{|x_1|\leq R_0}\partial_{x_2}(w_2-w_1)\ dx_1+\Ocal(1)\\
		&=2\int_{|x_1|\leq R_0}\frac{1}{\delta(x_1)}\ dx_1+\Ocal(1)\\
		&=2\int_{|x_1|\leq R_0}\frac{1}{\varepsilon+\lambda|x_1|^2}\ dx_1+\varepsilon^{\frac{\alpha-1}{2}}\Ocal(1)\\
		&=\frac{4}{\sqrt{\lambda\varepsilon}}\int_{0}^{R_0\sqrt{\frac{\lambda}{\varepsilon}}}\frac{1}{1+r^2}\ dr+\varepsilon^{\frac{\alpha-1}{2}}\Ocal(1)=\frac{2\pi}{\sqrt{\lambda\varepsilon}}+\varepsilon^{\frac{\alpha-1}{2}}\Ocal(1).
	\end{align*}

	\bibliographystyle{abbrv}
	\bibliography{Helmholtz_2D}{}
	
\end{document}